\documentclass[a4paper,11pt,reqno]{amsart}
\usepackage[centertags]{amsmath}
\usepackage{amsfonts,amssymb,amsthm} 
\usepackage[dvips]{graphicx} 
\usepackage{sidecap}
\usepackage{psfrag}
\usepackage[english]{babel}
\usepackage{newlfont}
\usepackage{color}
\usepackage{accents}
\usepackage[body={16cm,22cm},centering]{geometry} 
\usepackage{fancyhdr}
\usepackage{esint}

\usepackage[active]{srcltx}

\newtheorem{theorem}{Theorem}[section]

\newtheorem{lemma}[theorem]{Lemma}
\newtheorem{proposition}[theorem]{Proposition}

\theoremstyle{definition}
\newtheorem{remark}[theorem]{Remark}

\numberwithin{equation}{section}

\newcommand{\R}{\mathbb{R}}

\mathchardef\emptyset="001F

\title[Optimal Regularity and Free Boundary in Cohesive zone models]
{Optimal Regularity and Structure of the Free Boundary\\ for Minimizers in Cohesive zone models}

\author[L. Caffarelli, F. Cagnetti, A. Figalli]{L. Caffarelli, F. Cagnetti, A. Figalli}

\address[L. Caffarelli]{Department of Mathematics, The University of Texas at Austin, 1 University Station C1200, Austin TX 78712, USA}
\email{caffarel@math.utexas.edu}
\address[F. Cagnetti]
{University of Sussex, Pevensey 2, Department of Mathematics, 
BN1 9QH, Brighton, United Kingdom}
\email{f.cagnetti@sussex.ac.uk}
\address[A. Figalli]{Department of Mathematics, ETH Z\"urich, R\"amistrasse 101, 
8092 Zurich, Switzerland}
\email{alessio.figalli@math.ethz.ch}

\date{\today}

\begin{document}
\maketitle

\begin{center}
\begin{minipage}{11.5cm}
\small{
\noindent {\bf Abstract.} 
We study optimal regularity and free boundary 
for minimizers of an energy functional arising in 
cohesive zone models for fracture mechanics.
Under smoothness assumptions on the boundary conditions
and on the fracture energy density, we show that 
minimizers are $C^{1, 1/2}$, 
and that near non-degenerate points the fracture set is $C^{1, \alpha}$, for some $\alpha \in (0, 1)$.

\vspace{10pt}

\vspace{6pt}
}
\end{minipage}
\end{center}

\bigskip

\begin{section}{Introduction}

In recent years, a variational formulation of fracture evolution has been proposed 
by Francfort and Marigo \cite{FrMa98}, 
and later developed by Dal Maso and Toader \cite{DT02}, and
Dal Maso, Francfort, and Toader \cite{DMFT, DMFT2} 
(see also \cite{Mielhand} and the references therein, for 
a variational theory of rate independent processes).
Such evolution is based on the idea that at any given time 
the configuration of the elastic body is an absolute minimiser 
of the energy functional (see also \cite{ACFS, C, DalToa02}, 
and \cite{DDMM} in the context of plasticity 
where, more in general, critical points of the energy are allowed).

%
In this paper we study optimal regularity and free boundary 
for minimizers of an energy functional arising in 
cohesive zone models for fracture mechanics.
Such models describe the situation in which the energy density 
of the fracture depends on the distance between the lips of the crack
(see for instance \cite{ACFS, C, CT, CLO, DMZ}). 
We consider the energy functional associated to an elastic body occupying
the open strip $\mathbb{R}^n \times (-A, A)$, with $n \geq 2$ and $A > 0$.
Denoting a generic point $z \in \mathbb{R}^n \times (-A, A)$
by $(x, y)$, with $x \in \R^n$ and $y \in (-A, A)$, we shall consider deformations  ensuring that  
cracks can only appear on the hyperplane $\{ y = 0 \}$.
The assumption of confining fractures to a given hyperplane is 
a standard simplification that avoids some technical difficulties but does not prevent the crack set from being irregular, thus keeping the main features of the problem.

We consider the situation in which the elastic body 
can only undergo deformations that are parallel
to a fixed given direction lying on $\{ y = 0 \}$.
In this way, the displacement can be represented by a scalar function 
$\textup{v}: \mathbb{R}^n \times (-A, A) \to \R$.
According to Barenblatt's cohesive zone model \cite{Bar62}, 
the energy associated to a displacement $\textup{v} \in H^1 (\mathbb{R}^n \times (-A, A) \setminus \{ y =  0 \})$ is given by  
\begin{equation} \label{energian}
E(\textup{v}):=\frac{1}{2}\int_{\R^n \times (-A, A) \setminus \{ y = 0\}} | \nabla \textup{v} |^2 dz
+ \int_{\R^n} g (| [\textup{v}] |) \, d x.
\end{equation}
Here, $[\textup{v}] = \textup{v}_{\scriptscriptstyle RT} - \textup{v}_{\scriptscriptstyle LT}$, 
where $\textup{v}_{\scriptscriptstyle RT}$ and $\textup{v}_{\scriptscriptstyle LT}$ are the right and left traces on $\{ y=  0 \}$
of $\textup{v} \mid_{\R^n \times (0, A)}$ and 
$\textup{v} \mid_{\R^n \times (-A, 0)}$, respectively, 
and   
$g \in C^2[0, \infty)\cap C^3 (0, \infty)$ is strictly increasing, bounded,
with $g (0) = 0$ and $g'(0^+) \in (0,+\infty)$.
The parameter $g' (0^+)$ has an important physical meaning, 
and it can be identified with the \textit{maximal sustainable stress} 
of the material along $\{ y =  0\}$, see \cite[Theorem~4.6]{C}.
A critical point $u$ of \eqref{energian} with boundary conditions
$u_{\scriptscriptstyle A}, u_{\scriptscriptstyle -A}$ satisfies (see \cite[Proposition~3.2]{C}):
\begin{equation*} 
\begin{cases}
\Delta u = 0  \, &\textnormal{ in } \, \R^{n} \times (-A, A) \setminus \{ y= 0\}, \\
u =  u_{\scriptscriptstyle A} \, &\textnormal{ on } \{ y =  A \}, \\
u =  u_{\scriptscriptstyle -A} \, &\textnormal{ on } \{ y =  - A \}, \\
\partial_{y} u_{\scriptscriptstyle RT}= \partial_{y} u_{\scriptscriptstyle LT} \, &\textnormal{ on } \{ y= 0\}, \\
|\partial_{y} u | \leq g' (0^+) \, 
&\textnormal{ on } \{ y= 0\}, \\
\partial_{y} u = g' (|[u]|) \, \textnormal{sgn} ([u])
\, &\textnormal{ on } \{ y = 0 \} \cap \{ [u] \neq 0 \},
\end{cases}
\end{equation*}
where $\textnormal{sgn} (\cdot)$ denotes the sign function. Note that, because $\partial_{y} u_{\scriptscriptstyle RT} = \partial_{y} u_{\scriptscriptstyle LT}$, we can use the notation $\partial_yu$ to denote the $y$ derivative of $u$ on $\{y=0\}$ without paying attention to the side on which the derivative is computed.

For simplicity, we will assume $u_{\scriptscriptstyle A} (x) = - u_{\scriptscriptstyle -A}(x)$ for every $x \in \R^n$,
and we will focus on solutions that are odd 
with respect to the hyperplane $\{ y = 0 \}$.
In this situation, our problem reduces to the study 
of a function $u \in H^1 (\R^n \times (0,A))$ satisfying
\begin{equation} \label{defcritptn}
\begin{cases}
\Delta u = 0  \, &\textnormal{ in } \mathbb{R}^n \times (0,A), \\
u =  u_{\scriptscriptstyle A} \, &\textnormal{ on } \{ y =  A \}, \\
|\partial_{y} u | \leq g' (0^+) \, 
&\textnormal{ on } \{ y = 0 \}, \\
\partial_{y} u = g' (2|u|) \, \textnormal{sgn} (u) 
\, &\textnormal{ on }\{ y = 0 \} \cap \{ u \neq 0 \},
\end{cases}
\end{equation}
where we used the notation $u(x, 0) = u_{\scriptscriptstyle RT} (x, 0)$ for every $x \in \R^n$.
In this setting, the crack $K_u$ is represented by the discontinuity set of $u$,
and is given by 
\begin{equation} \label{fracture}
K_u := \{ (x, 0) : x \in \R^n, \, u (x, 0) \neq 0 \} \subset \R^n.
\end{equation}
We assume that the boundary condition $u_{\scriptscriptstyle A}$ satisfies the following:
\begin{equation} \label{assumptions uA}
u_{\scriptscriptstyle A} \in 
H^{1/2} (\mathbb{R}^n) \cap  C^{2, \beta} (\R^n) \text{ for some } \beta \in (0,1) \quad \text{ and } \quad 
\lim_{|x| \to \infty} u_{\scriptscriptstyle A} (x) = 0.
\end{equation}
Under these assumptions, we want to the study optimal regularity of 
the restriction of $u$ to $\R^n \times [0,A]$, and the regularity 
of the free boundary $\partial K_u$ (where the boundary is defined in the topology of $\mathbb{R}^n \times \{ 0 \}$).

\medskip

A major obstacle to the regularity of solutions is the possible presence 
of fracture points where $u$ changes sign.
Indeed, at such points the normal derivative $\partial_y u (\cdot, 0)$
is discontinuous with a jump of $2 g' (0^+)$,
due to the term $\text{sgn} (u)$ appearing in \eqref{defcritptn}. 
Our main contribution in this paper is to show that this possibility can never occur.

Problems of this type, where two phases (in this case the sets 
$\{ x \in \R^n : u (x,0) > 0 \}$
 and $\{ x \in \R^n : u (x,0) < 0 \}$)
can ``touch'' at a lower dimensional free boundary, 
have recently been studied by 
Allen and Petrosyan,  \cite{AP}, Allen \cite{A},
and  Allen, E. Lindgren and A. Petrosyan \cite{ALP}.
In these papers, to show the separation of phases
the authors use in a clever way the Alt-Caffarelli-Friedman and the Weiss monotonicity formulas.
Our approach is different and, although it requires 
$g$ to be sufficiently smooth, it  does not rely on monotonicity formulas.
Therefore, it can be applied to problems 
where the Laplacian is replaced by more general operators.



To show the separation of phases, we begin by proving some interesting general properties of the solutions, such as the fact that the crack set $K_u$ is bounded 
(see Lemma~\ref{infinitesimal} and Proposition~\ref{crackset}).
After that, we prove that certain regularity properties of $u_{\scriptscriptstyle A}$ ``propagate'' 
to $u (\cdot, y)$.  More precisely, for every $y \in [0, A)$ 
we prove that $u (\cdot, y)$ is Lipschitz continuous, 
that $u^+ (\cdot, y):= \max\{ u (\cdot, y), 0 \}$ is semiconvex, 
and that $u^- (\cdot, y):= \min\{ u (\cdot, y), 0 \}$ is semiconcave,
see Lemma~\ref{Lip2}, Lemma~\ref{semiconc2}, and
Lemma~\ref{semiconc3}.
Let us mention that, to show these regularity properties, we need to assume
$2 \| g'' \|_{L^{\infty}} < 1/A$. That is, we need the size $A$ of the strip 
to be sufficiently small, once the elastic properties of the material are given. 
As shown in Lemma \ref{lem:min}, under this assumption critical points are unique 
and therefore coincide with the global  minimizer.
We think this bound to be sharp, and this is in agreement with 
an explicit example given in \cite[Theorem~9.1 and Theorem~9.2]{C}, 
where uniqueness fails if $2 \| g'' \|_{L^{\infty}} > 1/A$.

Actually, in Lemmata~\ref{semiconc2} and~\ref{semiconc3} we prove a stronger property than the semiconvexity (resp. semiconcavity) of $u^+$ (resp. $u^-$),
since we need an estimate that allows us to ``connect'' the behavior of $u^+$ and $u^-$ near the set $\{u=0\}$
(see Remark \ref{useful phases sep}).
This plays indeed a crucial role in the proof of Proposition~\ref{phaseseparation}, where 
we prove that the two phases $\{ x \in \R^n : u (x,0) > 0 \}$
and $\{ x \in \R^n : u (x,0) < 0 \}$ are well separated.
We achieve this  in the following way:
First of all, exploiting Remark \ref{useful phases sep},
we prove that if
$(\overline{x}, 0) \in \partial K_u$ is any free boundary point 
where the sign of $u$ changes, then $u(\cdot, 0)$ is differentiable at $\overline x$
and $\nabla_x u (\overline{x}, 0) =  0$.
This, in turn, allows us to construct some   suitable barriers
from which we reach a contradiction.

Once we know that the sets $\{ u > 0\} \cap \{ y = 0 \}$ 
and $\{ u > 0 \} \cap \{ y = 0 \}$ are well separated, 
we can adapt to our setting the arguments used in \cite{AC, CF} to prove the optimal regularity of solutions:
\begin{theorem} \label{opt_reg_theorem}
Let $u_{\scriptscriptstyle A}$ satisfy \eqref{assumptions uA}, and let 
$g \in C^2[0, \infty) \cap C^3 (0, \infty)$ be strictly increasing and bounded,
with $g(0)=0$ and $g'(0^+) \in (0,+\infty)$.
Suppose, in addition, that $2 \| g'' \|_{L^{\infty}} < 1/A$, 
$\| g''' \|_{L^{\infty}} < \infty$,
and that $u \in H^1 (\R^n \times (0,A))$ is a solution of \eqref{defcritptn}. 
Then, $u \in C^{1,1/2}$.
\end{theorem}

\medskip
Once both phase separation and optimal regularity of $u$ are obtained, we deal with the regularity of the free boundary.
To this aim, we proceed by applying more standard techniques,
which are specific to operators for which monotonicity formulas are available.
Assuming without loss of generality that we are at a free boundary point coming from the positive phase, we subtract from $u$ the linear function $g' (0^+) y$, 
and then we reflect evenly with respect to the hyperplane $\{ y = 0\}$, defining
\begin{equation} \label{v intro}
v (x, y):=
\begin{cases}
u (x, y) - g' (0^+) y & \text{ for every }(x, y) \in \R^n \times (0, A), \\
v (x, -y) & \text{ for every } (x, y) \in \R^n \times (-A, 0).
\end{cases}
\end{equation}
Then, inspired by \cite{CSS}, we prove a variant of Almgren's monotonicity formula.
More precisely, suppose that $(0, 0) \in \partial K_u$, and set 
$$
\Phi_v (r):= r \frac{d}{dr} \log \left( \max\{ F_v (r) , r^{n+4} \} \right),
\qquad 
F_v (r) := \int_{\partial B_r} v^2 d \mathcal{H}^n,
$$ 
where $B_r$ is the ball of $\R^{n+1}$ centred at $0$ with radius $r$,
and $\mathcal{H}^n$ denotes the Hausdorff $n$-dimensional measure.
We show that there exists $C > 0$ such that for $r$ sufficiently small 
the function $r \mapsto \Phi_v (r) e^{C r}$ is nondecreasing
(see Proposition~\ref{9}).
This implies that $\Phi_v (0^+)$ exists, and we can show that either 
$\Phi_v (0^+) = n+3$, or $\Phi_v (0^+) \geq n+4$ (see Proposition~\ref{Phi}).
This allows us to classify subquadratic blow up profiles of $v$: more precisely, considering the family $\{v_r\}_{r>0}$ of functions 
$$
v_r (z) := \frac{v (r z)}{d_r}, \qquad \qquad d_r:= \left( \frac{F_v (r)}{r^n} \right)^{1/2},
$$
we can classify the possible limits as $r \to 0^+$ provided $\frac{d_r}{r^2}\to +\infty$.

In other words, provided $v$ decays slower than quadratic, we obtain the following theorem, 
which is the second main result of the paper.
\begin{theorem} \label{free_bound_theorem}
Let the assumptions of Theorem~\ref{opt_reg_theorem} be satisfied, 
and let $u \in H^1 (\R^n \times (0,A))$ be a solution of \eqref{defcritptn}. 
Suppose that $(0,0) \in \partial K_u$, with $u (\cdot, 0) \geq 0$ near $(0,0)$,
and let $v$ be defined by \eqref{v intro}. If
\begin{equation}
\label{eq:superquadr}
\liminf_{r \to 0^+} \frac{d_r}{r^2} = + \infty,
\end{equation}
then the free boundary 
$\partial K_u$ is of class $C^{1, \alpha}$ near $(0,0)$, for some $\alpha \in (0,1)$.
\end{theorem}
To prove Theorem~\ref{free_bound_theorem} we show that \eqref{eq:superquadr} implies that $\Phi_v$ attains its smallest possible value, namely $\Phi_v (0^+) = n +3$, and that in this case
blow up profiles of $v$ are homogeneous solutions of the classical Signorini problem 
(i.e. the classical thin obstacle problem), with homogeneity degree $1/2(\Phi_v (0^+) - n)$. 
Thanks to this fact, the blow ups can be easily classified
(see Proposition~\ref{blow up limit}) and the result follows as in the classical theory.\\

The paper is organised as follows. 
In Section 2 we introduce the notation and the setting of the problem.
We show basic regularity properties of the solution $u$ in Section 3, 
while Section~4 is devoted to the separation of phases and the optimal regularity.
Frequency formula is the subject of Section~5,
and in Section 6 we study blow up profiles.
Finally, in Section 7 we prove the regularity of the free boundary.
\end{section}

\begin{section}{Notation}

In this brief section we introduce the notation that will be used, and we give the main assumptions.
Throughout the paper, we fix $n \in \mathbb{N}$, with $n \geq 2$, and $A > 0$.
For every point $z \in \R^{n} \times [-A, A]$ we will write $z = (x, y)$, with $x \in \R^n$ and $y \in [-A, A]$.
The canonical basis of $\R^{n+1}$ is denoted by $\mathbf{e}_1,\dots, \mathbf{e}_{n+1}$. 
For $a, b \in \R^{n+1}$, $a \cdot b$ denotes the Euclidean scalar product between $a$ and $b$,
and $|\cdot|$ denotes both the absolute value in $\R$ and the Euclidean norm in $\R^n$ or $\R^{n+1}$, 
 depending on the context.
For every $k \in \mathbb{N}$, $\mathcal{H}^k$ stands for the Hausdorff $k$-dimensional measure.
If $z = (x, y) \in \R^{n+1}$ and $r> 0$, we will denote by $B_r (z)$ 
the ball of $\R^{n+1}$ centered at $z$ with radius $r$:
$$
B_r (z) = \{ \overline{z} \in \R^{n+1} : | \overline{z} - z | < r \},
$$
and with $B^n_r (x)$ the ball of $\R^{n}$ centered at $x$ with radius $r$:
$$
B^n_r (x) = \{ \overline{x} \in \R^{n} : | \overline{x} - x | < r \}.
$$
We will write $B_r$ and $B^n_r$ for $B^n_r (0)$ and $B^n_r (0)$, respectively, 
and we will use the notation $\mathbb{S}^{n}:= \partial B_1$ and
$\mathbb{S}^{n-1}:= \partial B^n_1$, while
$\omega_{n+1}$ denotes the $(n+1)$-dimensional Lebesgue measure of $B_1$. 

Throughout all the paper, $C$ will denote a universal constant, possibly different 
from line to line.
For any function $\textup{v} \in H^1 ( \R^n \times (-A, A) \setminus \{ y = 0 \})$, 
we will denote by $\textup{v}_{\scriptscriptstyle RT}$ and $\textup{v}_{\scriptscriptstyle LT}$ the right and left traces on $\{y=0\}$ of  $\textup{v}\mid_{\R^n \times (0, A)}$ and $\textup{v}\mid_{\R^n \times (-A, 0)}$, respectively, 
while we set
$$
\textup{v}^+ := \max\{ \textup{v}, 0 \} \quad \text{ and }  \quad \textup{v}^- := \min\{ \textup{v}, 0 \}, 
$$
so that $\textup{v} = \textup{v}^+ + \textup{v}^-$.
When \textup{v} is sufficiently regular, $\nabla \textup{v}$ and  $D^2 \textup{v}$
stand for the gradient and the Hessian of $\textup{v}$, while 
$\nabla_x \textup{v}$ and  $D^2_{xx} \textup{v}$ are the gradient and the Hessian 
of the function $x \mapsto \textup{v} (x, y)$. 
We will say that \textup{v} is \textit{homogeneous} of degree $\mu$ 
if \textup{v} can be written as
$$
\textup{v} (z) = |z|^{\mu} h \left( \frac{z}{|z|} \right), 
$$
for some function $h : \mathbb{S}^{n} \to \R$.
Let $L_0, D_0 \geq 0$.
For a function $f : \mathbb{R}^n \to \mathbb{R}$,
we say that $f$ is Lipschitz continuous, with Lipschitz constant $L_0$, if
$$
\sup_{x_1\neq x_2} \frac{|f (x_2) - f (x_1)|}{|x_2 - x_1|} \leq L_0.
$$
Also, $f$ is said to be \textit{semiconvex}, with semiconvexity constant $D_0$, if
$$
f (x + h) + f (x - h) - 2 f (x) \geq - D_0 |h|^2, 
$$
for every $x,h \in \mathbb{R}^n$.
Similarly, we say that $f$ is \textit{semiconcave}, with semiconcavity constant $D_0$, if
$$
f (x + h) + f (x - h) - 2 f (x) \leq  D_0 |h|^2, 
$$
for every $x,h \in \mathbb{R}^n$.

\medskip

We are now ready to state our assumptions.
In the following, $g \in C^2[0, \infty) \cap C^3 (0, \infty)$ is strictly increasing 
and bounded, with $g(0)=0$ and $g'(0^+) \in (0,+\infty)$.
We assume, in addition, that $2 \| g'' \|_{L^{\infty}} < 1/A$ 
and $\| g''' \|_{L^{\infty}} < \infty$, where $\| g'' \|_{L^{\infty}}$ and $\| g''' \|_{L^{\infty}}$ denote the $L^{\infty}$-norms of $g''$ and $g'''$, respectively.
Moreover, we assume that $u_{\scriptscriptstyle A} : \R^n \to \R$ satisfies \eqref{assumptions uA}, i.e.
$$
u_{\scriptscriptstyle A} 
\in H^{1/2} (\mathbb{R}^n) \cap   C^{2, \beta} (\R^n) \text{ for some } \beta \in (0,1) \quad \text{ and } \quad 
\lim_{|x| \to \infty} u_{\scriptscriptstyle A} (x) = 0.
$$
\begin{remark} \label{DLuA}
The assumptions above imply, in particular,  that 
$u_{\scriptscriptstyle A}$ is Lipschitz continuous with Lipschitz constant $L_A:=\|\nabla u_{\scriptscriptstyle A}\|_{L^\infty}$.
Moreover, denoting by $\lambda_{\text{min}} (x)$ and $\lambda_{\text{max}} (x)$
the smallest and largest eigenvalue of $D^2 u_{\scriptscriptstyle A} (x)$, respectively,  
we have that $u_{\scriptscriptstyle A}$ is semiconvex with semiconvexity constant 
$D_A := \| (\lambda_{\text{min}})^-\|_{L^\infty}$, and is semiconcave with semiconcavity constant
$C_A := \| (\lambda_{\text{max}})^+\|_{L^\infty}$. 
\end{remark}
We will study optimal regularity and free boundary for a function $u \in H^1 (\R^n \times (0,A))$
solving equation \eqref{defcritptn}:
$$
\begin{cases}
\Delta u = 0  \, &\textnormal{ in } \mathbb{R}^n \times (0,A), \\
u =  u_{\scriptscriptstyle A} \, &\textnormal{ on } \{ y =  A \}, \\
|\partial_{y} u | \leq g' (0^+) \, 
&\textnormal{ on } \{ y = 0 \}, \\
\partial_{y} u = g' (2|u|) \, \textnormal{sgn} (u) 
\, &\textnormal{ on }\{ y = 0 \} \cap \{ u \neq 0 \}.
\end{cases}
$$
Note that the equation above implies that 
\begin{equation} \label{normalderiv}
- g' (2|u (x,0)|) \leq \partial_{y} u (x,0) \leq g' (2|u(x,0)|) \quad \textnormal{ for every } x \in \mathbb{R}^n. 
\end{equation}
Also, by the maximum principle,
$$
\|u\|_{L^\infty}\leq \|u_{\scriptscriptstyle A}\|_{L^\infty}<\infty.
$$
In the next section we prove some basic regularity properties of $u$.

\end{section}

\begin{section}{Basic properties of the solution}

We study in this section the basic regularity properties 
of a solution $u$ of equation \eqref{defcritptn}. 
We start by showing that condition $2 \| g'' \|_{L^{\infty}} < 1/A$ 
implies uniqueness.
\begin{lemma}
\label{lem:min}
Let $u_{\scriptscriptstyle A}$ satisfy \eqref{assumptions uA}, and let 
$g \in C^2[0, \infty)$ be strictly increasing and bounded,
with $g(0)=0$ and $g'(0^+) \in (0,+\infty)$.
If $2 \| g'' \|_{L^{\infty}} < 1/A$, 
then there exists a unique $u \in H^1 (\R^n \times (0,A))$ 
solving~\eqref{defcritptn}.
In particular, there is a unique critical point of \eqref{energian} that coincides with the global minimizer.
\end{lemma}
\begin{proof}
Suppose, by contradiction, that there exist $u_1, u_2 \in H^1 (\R^n \times (0,A))$
solutions of \eqref{defcritptn}, with $u_1 \not \equiv u_2$.
In particular, since $u_1 = u_2$ on $\{ y= A \}$, this implies
\begin{equation} \label{l2normgradient}
\| \nabla (u_1 - u_2) \|^2_{L^2 (\R^n \times (0, A))} > 0.
\end{equation}
We will prove the statement into two steps.

\vspace{.2cm}

\textbf{Step 1:} We show that
\[
\| \nabla (u_1 - u_2) \|^2_{L^2 (\R^n \times (0, A))}
 \leq 2 \| g'' \|_{L^{\infty}} \| u_1 - u_2 \|^2_{L^2(\R^n)}.
\]
Using the weak formulation of the equation (see \cite[Proposition 3.1]{C}) we have
\begin{equation} \label{weak solution}
\int_{\R^n \times (0, A)} \nabla u_1 \cdot \nabla \psi \, dz +
\int_{\R^n} \Big( \psi \, g' (2 |u_1|) \, \textnormal{sgn} (u_1)
\, 1_{\{ u_1 \neq 0 \}} + g' (0^+) | \psi | 1_{\{ u_1 = 0 \}} \Big) \, d x \geq 0,
\end{equation}
for every $\psi \in H^1 (\R^n \times (0, A))$ with $\psi = 0$ on $\{ y =  A \}$.
Choosing $u_2 - u_1$ as test function in \eqref{weak solution} we obtain
\begin{align*}
&\int_{\R^n \times (0, A)} \nabla u_1 \cdot \nabla (u_2 - u_1) \, dz \\
&\hspace{.4cm}+\int_{\R^n} \Big( (u_2 - u_1) \, g' (2 | u_1 |) \, \textnormal{sgn} (u_1)
\, 1_{\{ u_1 \neq 0 \}} + g' (0^+) |u_2 - u_1| 1_{\{ u_1 = 0 \}} \Big) \, d x \geq 0.
\end{align*}
Analogously, using the weak formulation of the equation for $u_2$, with test function $u_1 - u_2$, we get
\begin{align*}
&\int_{\R^n \times (0, A)} \nabla u_2 \cdot \nabla (u_1 - u_2) \, dz \\
&\hspace{.4cm}+\int_{\R^n} \Big( (u_1 - u_2) \, g' (2 | u_2 |) \, \textnormal{sgn} (u_2)
\, 1_{\{ u_2 \neq 0 \}} + g' (0^+) |(u_1 - u_2)| 1_{\{ u_2 = 0 \}} \Big) \, d x \geq 0.
\end{align*}
Adding together the last two relations, we obtain
\begin{align*}
&\| \nabla (u_1 - u_2) \|^2_{L^2 (\R^n \times (0, A))} \\
&\hspace{.4cm}\leq \int_{\R^n} \Big( (u_2 - u_1) \, g' (2 | u_1 |) \, \textnormal{sgn} (u_1)
\, 1_{\{ u_1 \neq 0 \}} + g' (0^+) |u_2 - u_1| 1_{\{ u_1 = 0 \}} \Big) \, d x \\
&\hspace{.4cm}+ \int_{\R^n} \Big( (u_1 - u_2) \, g' (2 | u_2 |) \, \textnormal{sgn} (u_2)
\, 1_{\{ u_2 \neq 0 \}} + g' (0^+) |(u_1 - u_2)| 1_{\{ u_2 = 0 \}} \Big) \, d x \\
&\hspace{.4cm}=\int_{\R^n} 
(u_2 - u_1) \, \Big(  g' (2 |u_1|) \, \textnormal{sgn}(u_1) 
- g' (2 |u_2|) \, \textnormal{sgn}(u_2) \Big)
\, 1_{\{ u_1 u_2 \neq 0 \}} \, d x \\
&\hspace{.4cm}+ \int_{\R^n} 
| u_2 | \Big( g' (0^+) -  g' (2 | u_2 |) \Big)       
\, 1_{\{ u_1 = 0 \} \cap \{ u_2 \neq 0 \}}  \, d x \\
&\hspace{.4cm}+ \int_{\R^n} 
| u_1 | \Big(  g' (0^+) -  g' ( 2 | u_1| ) \Big)       
\, 1_{\{ u_1 \neq 0 \} \cap \{ u_2 = 0 \}}  \, d x.
\end{align*}
We observe now that 
$$
(u_2 - u_1) \, \Big(  g' (2 |u_1|) \, \textnormal{sgn} (u_1) 
- g' (2 |u_2|) \, \textnormal{sgn}(u_2) \Big) < 0 \quad \text{ whenever } u_1 u_2 <0,
$$
therefore
\begin{align*}
&\| \nabla (u_1 - u_2) \|^2_{L^2 (\R^n \times (0, A))} \\
&\hspace{.4cm}\leq  \int_{\R^n} | u_2 - u_1  | \, |  \, g' (2 |u_1|) 
- g' (2 | u_2 |) \, |
\, 1_{ \{ u_1 u_2 > 0 \} } \, d x \\
&\hspace{.4cm}+ \int_{\R^n} 
| u_2 | \Big( g' (0^+) -  g' (2 | u_2 |) \Big)       
\, 1_{\{ u_1 = 0 \} \cap \{ u_2 \neq 0 \}}  \, d x \\
&\hspace{.4cm}+ \int_{\R^n} 
| u_1 | \Big(  g' (0^+) -  g' (2 |u_1 |) \Big)       
\, 1_{\{ u_1 \neq 0 \} \cap \{ u_2 = 0 \}}  \, d x \\
&\hspace{.4cm}\leq 2 \int_{\R^n} 
\| g'' \|_{L^{\infty}} | u_1 - u_2|^2 
1_{ \{ u_1 u_2  >0 \} } \, d x + 2 \int_{\R^n} 
\| g'' \|_{L^{\infty}} | u_1 - u_2|^2 
1_{\{ u_1 = 0 \} \cap \{ u_2 \neq 0 \}} \, d x \\
&\hspace{1 cm}+2 \int_{\R^n} 
\| g'' \|_{L^{\infty}} | u_1 - u_2 |^2 
1_{\{ u_1 \neq 0 \} \cap \{ u_2 = 0 \}} \, d x \\
&\hspace{.5 cm} \leq 2 \| g'' \|_{L^{\infty}} \|  u_1 - u_2 \|^2_{L^2(\R^n)},
\end{align*}
where we also used the fact that
$| u_2  - u_1  |= |\, | u_2 | - | u_1 | \,|$ whenever $ u_1 u_2 >0$. 

\vspace{.2cm}

\textbf{Step 2:} We conclude.

First of all, note that
\[
u_{\scriptscriptstyle A} (x) = u_i (x, A) = u_i  (x, 0) + \int_0^A \partial_y u_i (x, t) \, dt \qquad \text{ for every }x \in \R^n \text{ and }i = 1, 2.
\]
Therefore, for every $x \in \R^n$, 
\begin{align*}
u_2 (x, 0) - u_1 (x, 0) = \int_0^A \partial_y (u_1 - u_2) (x, t) \, dt
\leq A^{1/2} \left( \int_0^A | \nabla (u_1 - u_2) (x, t) |^2 \, dt \right)^{1/2}, 
\end{align*}
so that
\[
\| u_1 - u_2 \|^2_{L^2(\R^n)} \leq A \| \nabla (u_1 - u_2) \|^2_{L^2 (\R^n \times (0, A))}.
\]
Then, thanks to Step 1
\begin{align*}
\| \nabla (u_1 - u_2) \|^2_{L^2 (\R^n \times (0, A))}
&\leq 2 A \| g'' \|_{L^{\infty}} \| \nabla (u_1 - u_2) \|^2_{L^2 (\R^n \times (0, A))}.
\end{align*}
Since $2 A \| g'' \|_{L^{\infty}} < 1$, this implies 
\[
\| \nabla (u_1 - u_2) \|^2_{L^2 (\R^n \times (0, A))} = 0, 
\]
against \eqref{l2normgradient}.

\end{proof}
We now show that $x \longmapsto u (x, 0)$ is infinitesimal as $|x| \to \infty$.
\begin{lemma} \label{infinitesimal}
Let $u_{\scriptscriptstyle A}$ and $g$ be as in Theorem~\ref{opt_reg_theorem}, 
and let $u \in H^1 (\R^n \times (0,A))$ be a solution of \eqref{defcritptn}.
Then, 
$$
\lim_{|x| \to \infty} u (x,0) = 0.
$$
\end{lemma}

\begin{proof}
Suppose, by contradiction, that there exists a sequence 
$\{ x_k \}_{k \in \mathbb{N}} \subset \mathbb{R}^n$ such that 
$|x_k| \to \infty$ and
$$
\lim_{k \to \infty} u (x_k,0) = a \neq 0.
$$
Define now, for every $k \in \mathbb{N}$, the function $u_k: \mathbb{R}^n \times [0,A] \to \mathbb{R}$ as 
$$
u_k (x, y) : = u (x + x_k, y).
$$
Since $u_k$ is harmonic for every $k \in \mathbb{N}$ 
and $\{ u_k \}_{k \in \mathbb{N}}$ is uniformly bounded in $\mathbb{R}^n \times [0,A]$ and $\|u_k\|_{H^1(\R^n\times (0,A))}\leq C$,
up to subsequences we have 
\begin{equation} \label{unif_conv}
u_k \to \overline{u} \quad \text{ uniformly on compact subsets of } \mathbb{R}^n \times [0,A],
\end{equation}
for some harmonic function $\overline{u}: \mathbb{R}^n \times [0,A] \to \mathbb{R}$
such that $\overline{u} (\cdot, A) \equiv 0$ and $\overline{u} (0,0) = a$, with 
$\overline u \in H^1 (\R^n \times (0,A))$.
Since $u_k$ is harmonic for each $k$, we have
\begin{align}
0 &= \int_{\mathbb{R}^n \times (0,A)} {u}_k \Delta u_k \, dz
= \int_{\mathbb{R}^n \times (0,A)} \text{div} ({u}_k \nabla u_k) \, dz
- \int_{\mathbb{R}^n \times (0,A)} |\nabla u_k|^2  \, dz \nonumber \\
&= \int_{\mathbb{R}^n} {u}_A (x+x_k) ( \partial_y u_k) (x, A) \, dx
- \int_{\mathbb{R}^n} {u}_k (x, 0) ( \partial_y u_k) (x, 0) \, dx
- \int_{\mathbb{R}^n \times (0,A)} |\nabla u_k|^2 \, dz \nonumber \\
&= \int_{\mathbb{R}^n} {u}_A (x+x_k) ( \partial_y u_k) (x, A) \, dx
- \int_{\mathbb{R}^n} |{u}_k (x, 0)|\,g'(2|{u}_k (x, 0)|) \, dx
- \int_{\mathbb{R}^n \times (0,A)} |\nabla u_k|^2  \, dz\nonumber\\
&\leq \int_{\mathbb{R}^n} {u}_A (x+x_k) ( \partial_y u_k) (x, A) \, dx
- \int_{\mathbb{R}^n \times (0,A)}|\nabla u_k|^2  \, dz.
 \label{weak}
\end{align}
Letting $k\to \infty,$ since $u_{\scriptscriptstyle A}(x_k+\cdot)\to 0$, we obtain
$$
\int_{\mathbb{R}^n \times (0,A)} | \nabla \overline{u} |^2 \, dz
\leq \liminf_{k\to \infty} \int_{\mathbb{R}^n \times (0,A)}|\nabla u_k|^2  \, dz= 0,
$$
where we also used the fact that $u_k \rightharpoonup \overline{u}$ weakly in $H^1_{loc} (\mathbb{R}^n \times (0,A))$.
Since $\overline u(\cdot,A)\equiv 0$ this implies  $\overline u \equiv 0$, which contradicts the fact that 
 $\overline{u} (0,0) = a \neq 0$.
\end{proof}
We now prove that the crack set $K_u$ defined in \eqref{fracture} is bounded.
\begin{proposition} \label{crackset}
Let $u_{\scriptscriptstyle A}$ and $g$ be as in Theorem~\ref{opt_reg_theorem}, 
and let $u \in H^1 (\R^n \times (0,A))$ be a solution of \eqref{defcritptn}. Then, 
$u (\cdot, 0)$ has compact support.
\end{proposition}

\begin{proof}
We start by showing that there exist positive constants $R = R (g, A)$, 
$c = c(g, A)$, and $r = r (g, A) \in (0,1)$ with the following property:
If $x_1 \in \mathbb{R}^n$ is such that $|x_1| > R$ and $u(x_1,0) \neq 0$, then
\begin{equation} \label{step1}
\exists \, z_1 \in \overline{B_1^n} (x_1) \quad \text{ such that } \qquad \int_{B_r^n (z_1) \times (0,A)} |\nabla u|^2 \, dx \, dy \geq c.
\end{equation}
Before proving the claim, let us show that this implies the conclusion.
Indeed, suppose by contradiction that the support of $u (\cdot, 0)$ is not bounded. 
Then, there exists a sequence $\{ x_k \}_{k \in \mathbb{N}} \subset \mathbb{R}^{n}$ 
with $|x_k| \to \infty$ such that $|x_k| > R$ and $u(x_k,0) \neq 0$ for every $k \in \mathbb{N}$.
By \eqref{step1}, for every $k \in \mathbb{N}$ there exists 
$z_k \in \overline{B_1^n} (x_k)$ such that
$$
\int_{B_r^n (z_k) \times (0,A)} |\nabla u|^2 \, dx \, dy \geq c.
$$
Without loss of generality, we can assume that $|x_j - x_k| \geq 4$ for every $j \neq k$, 
so that the balls $\{ B_r^n (z_k) \}_{k \in \mathbb{N}}$
are pairwise disjoint.
Therefore, 
$$
\int_{\mathbb{R}^{n} \times (0,A)} |\nabla u|^2 \, dx \, dy 
\geq \sum_{k = 1}^{\infty} \int_{B^n_r (z_k) \times (0,A)} |\nabla u|^2 \, dx \, dy 
= \infty,
$$
against the fact that $u \in H^1 (\R^n \times (0, A))$. 

Let us now show the claim. By Lemma \ref{infinitesimal}, 
\begin{equation} \label{tozero}
\lim_{|x| \to \infty} u_{\scriptscriptstyle A} (x) = \lim_{|x| \to \infty} u(x, 0) = 0.
\end{equation}
Let $V : \overline{B_1^n} \times [0,A] \to \mathbb{R}$
be the solution of the following problem:
$$
\begin{cases}
\Delta V = 0  \, &\textnormal{ in } \, B_1^n \times (0,A), \\
V = |x|^2 \, &\textnormal{ on } \, B_1^n \times \{ y = 0 \}, \\
V = 1 \, &\textnormal{ on } B_1^n \times \{ y =  A \}, \\
V = 1 \, &\textnormal{ on } \partial B_1^n \times (0,A),
\end{cases}
$$
and let $a = a (g, A) > 0$ be so small that
\begin{equation} \label{s}
\sup_{|x| \leq \frac{1}{2}} | \partial_y V (x,0)| < \frac{g' (0^+)}{2 a} 
\qquad
\text{ and } \qquad g'(s) > \frac{g' (0^+)}{2} \quad \text{ for } 0 < s < \frac{a}{2}.
\end{equation}
By \eqref{tozero}, there exists a constant $R=R(g, u_{\scriptscriptstyle A}) > 2$ such that
\begin{equation} \label{x1 small}
|u(x,0)| < \frac{a}{4} \quad \text{ and } \quad |u(x,A)| = |u_{\scriptscriptstyle A}(x)| < \frac{a}{4}, \quad 
\text{ for every } x \text{ with } |x| > R - 2.
\end{equation}
Let $x_1 \in \mathbb{R}^n$ be such that $|x_1| > R$ and $u (x_1,0) > 0$
(the case $u (x_1,0) < 0$ can be treated in the same way).
We will show that there exist $z_1 \in \overline{B_1^n} (x_1)$,
$c > 0$, and $r \in (0, 1)$ 
such that \eqref{step1} holds true.  

For every $b > 0$ define $V_b (x,y) := a V (x-x_1,y) + b$
and set
$$
\overline{b}:= \inf \{ b > 0 : V_b > u \text{ in } B^n_1 (x_1) \times (0,A) \}.
$$
Note that we necessarily have $\overline{b} > 0$, since $V_0 (x_1,0) = 0 < u (x_1,0)$.

By maximum principle, there exists 
$(\overline{x}, \overline{y}) \in \partial (B^n_1 (x_1) \times (0,A))$
such that
$$
V_{\overline{b}} (\overline{x}, \overline{y}) = u (\overline{x}, \overline{y}).
$$
By \eqref{x1 small} it follows that $\overline{y} \neq A$, since 
$u(x, A) < a/4 < a+ \overline{b} = V_{\overline{b}} (x, A)$ for every 
$x$ with $|x-x_1| \leq 1$.
We then have only two possibilities.

\medskip

\noindent
\textbf{Case i: $\overline{y}=0$.}
Let us show that this is not possible.
First of all, note that in this case it must be $|\overline{x}-x_1| \leq 1/2$.  
Indeed, for every $x \in B_1^n (x_1)$ with $|x-x_1| > 1/2$
$$
V_{\overline{b}} (x, 0) = a |x-x_1|^2 + \overline{b}
> \frac{a}{4} > u (x, 0),
$$
thanks to \eqref{x1 small}.
Thus, using \eqref{s} and the fact that $u (\overline{x}, 0) = V_{\overline{b}} (\overline{x}, 0) > 0$, 
we have
\begin{align*}
\frac{g' (0^+)}{2} < g' (2 u (\overline{x},0)) = \partial_y u (\overline{x},0)
\leq \partial_y V_{\overline{b}} (\overline{x},0) = a \, \partial_y V (\overline{x} - x_1,0)
< \frac{g' (0^+)}{2},
\end{align*}
which gives a contradiction.

\medskip

\noindent
\textbf{Case ii: $0< \overline{y} < A$ and $|\overline{x} - x_1|=1$.}
Let us show that, for $a$ sufficiently small, there exists a positive constant
$c_1 = c_1 (g, a, A, u_{\scriptscriptstyle A})$ such that
\begin{equation} \label{bar y}
0 < c_1 < \overline{y} < A-c_1 < A.
\end{equation}
From \eqref{defcritptn} and \eqref{assumptions uA} it follows that
$$
| \partial_y u (x, 0)| \leq  g' (0^+) \qquad | \partial_y u (x, A)| \leq C_0, 
\qquad \text{ for every } x \in \mathbb{R}^n,
$$
for some positive constant $C_0 > g' (0^+)$.
Therefore, setting $C_1:= (C_0 - g' (0^+))/A$, by the maximum principle  (note that $\partial_yu$ is harmonic)
\begin{equation} \label{bound partial_y}
- g' (0^+) - C_1 y  \leq \partial_y u (x,y) \leq g' (0^+) + C_1 y \quad
\text{ for every } (x,y) \in \mathbb{R}^n \times [0,A].
\end{equation}
Therefore, 
\begin{align*}
a < a + \overline{b} &= u (\overline{x} , \overline{y} ) = 
u (\overline{x} , 0 ) + \int_0^{\overline{y}} \partial_y u (\overline{x}, y) \, dy \\
&< \frac{a}{4} + \int_0^{\overline{y}} (g' (0^+) + C_1 y) \, dy 
= \frac{a}{4} + g' (0^+) \overline{y} + C_1 \frac{\overline{y}^2}{2}
\end{align*}
where we used \eqref{x1 small}.
The above inequality implies
\begin{equation*}
\overline{y} > \frac{- 2 g' (0^+) + \sqrt{4 g' (0^+)^2 + 6 a C_1}}{2 C_1} >  c_1>0.
\end{equation*}
Analogously, we have
\begin{align*}
a < a + \overline{b} &= u (\overline{x} , \overline{y} ) = 
u (\overline{x} , A ) - \int_{\overline{y}}^{A} \partial_y u (\overline{x}, y) \, dy \\
&\leq \frac{a}{4} + \int_{\overline{y}}^{A} (g' (0^+) + C_1 y) \, dy 
= \frac{a}{4} + g' (0^+) (A - \overline{y}) + C_1 \frac{(A- \overline{y})^2}{2}, 
\end{align*}
which implies $A - \overline{y} > c_1$, so that 
$$
\min\{ \overline{y} , A - \overline{y} \} > c_1,
$$
thus giving \eqref{bar y}.

We can now show \eqref{step1}.
At the contact point, we have
$$
u (\overline{x} , \overline{y} ) = V_{\overline{b}} (\overline{x} , \overline{y} ) = a + \overline{b}.
$$
Then, by Harnack inequality and by \eqref{bar y}, there exists a radius 
$r = r (c_1) \in (0, 1)$ such that
$$
u(x,y) \geq \frac{a}{2} \qquad \text{ for every } (x,y) \in B^n_r (\overline{x}) \times 
(\overline{y} - r , \overline{y} + r).
$$
The inequality above implies 
that for every $x \in B^n_r (\overline{x})$
(note that $|x| > R - 2$ for $x \in B^n_r (\overline{x})$, so we can use \eqref{x1 small})
$$
\frac{a}{4} \leq u(x,\overline{y}) - u(x,A) \leq \int_{\bar y}^A | \partial_y u |( x,y)\,dy 
\leq \sqrt{A} \left( \int_0^A (\partial_y u)^2(x,y)\,dy \right)^{\frac{1}{2}},
$$
from which
$$
\int_0^A |\nabla u|^2(x,y)\,dy \geq \frac{a^2}{16 A} 
\qquad \qquad \forall \, x \in B^n_r (\overline{x}).
$$
Integrating with respect to $x$, we obtain
$$
\int_{B^n_r(\overline{x}) \times (0,A)} |\nabla u|^2 \, dx \, dy
\geq \frac{a^2 \mathcal{H}^n (B^n_r(\overline{x}))}{16 A }.
$$
Setting 
$$
z_1:= \overline{x}, \qquad \text{ and }\qquad c  := \frac{a^2 \mathcal{H}^n (B^n_r(\overline{x}))}{16 A}, 
$$
the claim follows.
\end{proof}
We now show that, under the assumption $2A\| g '' \|_{L^{\infty}}<1,$ the Lipschitz continuity of $u_{\scriptscriptstyle A}$ 
implies the Lipschitz continuity of $u (\cdot, y)$, uniformly with respect to $y$.

\begin{lemma} \label{Lip2}
Let $u_{\scriptscriptstyle A}$ and $g$ be as in Theorem~\ref{opt_reg_theorem}, 
let $u \in H^1 (\R^n \times (0,A))$ be a solution of \eqref{defcritptn}, 
and let $L_A$ be given by Remark~\ref{DLuA}.
Then, for every $y \in [0,A]$ the function $u (\cdot, y)$ is Lipschitz continuous, with Lipschitz constant
$\frac{L_A}{1-2A\| g '' \|_{L^{\infty}}}$.
\end{lemma}

\begin{proof}
Let $h \in \mathbb{R}^n \setminus \{ 0 \}, \alpha > 0$, and define for every $C > 0$
$$
u^{h, \alpha}_C (x,y) : = u (x+h,y) + C |h| \left[ 1 + \alpha \left(1 - \frac{y}{A} \right) \right], \qquad (x,y) \in \R^n \times [0,A].
$$
Setting $C^{\alpha}_h: = \inf \left\{ C > 0 : u^{h, \alpha}_C > u \right\}$, 
we claim that 
\begin{equation} \label{C12}
C^{\alpha}_h \leq L_A, 
\qquad \text{ for }  \alpha  >\frac{2A\| g '' \|_{L^{\infty}}}{1- 2A\| g '' \|_{L^{\infty}}}.
\end{equation}
Let us first show that the claim proves the lemma.
Indeed, if \eqref{C12} is true then 
for every $(x,y) \in \mathbb{R}^n \times [0,A]$ we have 
\begin{align*}
u (x+h,y) + L_A(1+\alpha) |h|
&\geq u (x+h,y) + L_A |h| \left[ 1 + \alpha \left(1 - \frac{y}{A} \right) \right] \\ 
&\geq u (x+h,y) + C^{\alpha}_h |h| \left[ 1 + \alpha \left(1 - \frac{y}{A} \right) \right] \geq u (x, y).
\end{align*}
Since $x,y$ and $h$ are arbitrary, from the last inequality and letting $\alpha  \to \frac{2A\| g '' \|_{L^{\infty}}}{1- 2A\| g '' \|_{L^{\infty}}}$, we get
$$
|u (x+h,y) - u (x,y)| \leq \frac{L_A}{1-2A\| g '' \|_{L^{\infty}}}  |h| ,
$$
thus concluding.

Let us now prove the claim.
By maximum principle and thanks to \eqref{tozero}, there exists 
$(\overline{x}, \overline{y}) \in \R^n \times \{ 0, A\}$ such that
\begin{equation*} 
0 = u^{h, \alpha}_{C^{\alpha}_h} (\overline{x}, \overline{y}) - u (\overline{x}, \overline{y})
= \inf_{\R^n \times [0,A]} ( u^{h, \alpha}_{C^{\alpha}_h} - u ).
\end{equation*}
In the following we assume $C^{\alpha}_h > 0$, since otherwise \eqref{C12} is trivially satisfied. 
We have two possibilities.

\medskip

\noindent
\textbf{Case 1: $\overline{y} = A$.} 
Since $u_{\scriptscriptstyle A} (\cdot)$ is Lipschitz continuous, at the contact point $(\overline{x}, A)$ we have
$$
- L_A |h| \leq u_{\scriptscriptstyle A} (\overline{x}+h) - u_{\scriptscriptstyle A} (\overline{x}) = - C^{\alpha}_h |h| ,
$$
from which \eqref{C12} follows. 

\medskip

\noindent
\textbf{Case 2: $\overline{y} =0$.} We conclude the proof of the lemma, 
showing that for $\alpha$ sufficiently large this case is impossible.
At the contact point, the following equality holds true:
\begin{equation} \label{contact}
u (\overline{x}+h,0) + (1+ \alpha) C^{\alpha}_h |h| = u^{h, \alpha}_{C^{\alpha}_h} (\overline{x},0) = u (\overline{x},0).
\end{equation}
We consider now three possible subcases, in which we will always reach a contradiction.

\medskip

\noindent
\textbf{Case 2a: $\overline{y} =0$ and $u (\overline{x},0) \leq 0$.}
Thanks to \eqref{contact}, it has to be $u (\overline{x}+h,0) \leq - (1 + \alpha) C^{\alpha}_h |h| < 0$.
Therefore, recalling \eqref{normalderiv} we get
\begin{align*}
& - g' (- 2  u (\overline{x},0 ) ) \leq \partial_y u (\overline{x},0 ) \leq \partial_y u^{h, \alpha}_{C^{\alpha}_h} (\overline{x},0) \\
& = \partial_y u (\overline{x} + h,0) - \frac{\alpha C^{\alpha}_h |h|}{A} \\
&= - g' ( - 2 u (\overline{x} + h,0) ) - \frac{\alpha C^{\alpha}_h |h|}{A} \\
&= - g' \big(- 2  u (\overline{x},0 ) + 2 (1+ \alpha) C^{\alpha}_h |h| \big) - \frac{\alpha C^{\alpha}_h |h|}{A} \\
&\leq - g' (- 2  u (\overline{x},0 ) )  + C^{\alpha}_h |h| 
\left( 2 (1+ \alpha) \| g '' \|_{L^{\infty}} - \frac{\alpha}{A} \right) \\
&= - g' (- 2  u (\overline{x},0 ) )  + C^{\alpha}_h |h| 
\left[ 2\| g '' \|_{L^{\infty}} + \alpha \left( 2\| g '' \|_{L^{\infty}} - \frac{1}{A} \right) \right] \\
& <  - g' (- 2  u (\overline{x},0 ) ),
\end{align*}
for $\alpha>\frac{2A\| g '' \|_{L^{\infty}}}{1- 2A\| g '' \|_{L^{\infty}}}$.

\medskip

\noindent
\textbf{Case 2b: $\overline{y} =0$ with $u (\overline{x},0) > 0$ and $u (\overline{x}+h,0) < 0$.}
In this case we have
\begin{align*}
& 0 < g' (2  u (\overline{x},0 ) ) 
= \partial_y u (\overline{x},0 ) \leq \partial_y u^{h, \alpha}_{C^{\alpha}_h} (\overline{x},0) 
 = \partial_y u (\overline{x} + h,0) - \frac{\alpha C^{\alpha}_h |h|}{A} \\
&=  - g' ( - 2 u (\overline{x} + h,0) ) - \frac{\alpha C^{\alpha}_h |h|}{A} < 0,
\end{align*}
which is still impossible.

\medskip

\noindent
\textbf{Case 2c: $\overline{y} =0$ with $u (\overline{x},0) > 0$ and $u (\overline{x}+h,0) \geq 0$.}
This follows as in case 2a:
\begin{align*}
& g' (2  u (\overline{x},0 ) ) = \partial_y u (\overline{x},0 ) 
\leq \partial_y u^{h, \alpha}_{C^{\alpha}_h} (\overline{x},0) \\
& = \partial_y u (\overline{x} + h,0) - \frac{\alpha C^{\alpha}_h |h|}{A} 
\leq g' ( 2 u (\overline{x} + h,0) ) - \frac{\alpha C^{\alpha}_h |h|}{A} \\
&= g' \big(2  u (\overline{x},0 ) - 2 (1+ \alpha) C^{\alpha}_h |h| \big) - \frac{\alpha C^{\alpha}_h |h|}{A} \\
&\leq g' (2  u (\overline{x},0 ) )  + C^{\alpha}_h |h| 
\left( 2 (1+ \alpha) \| g '' \|_{L^{\infty}} - \frac{\alpha}{A} \right) \\
&= g' (2  u (\overline{x},0 ) )  + C^{\alpha}_h |h| 
\left[ 2\| g '' \|_{L^{\infty}} + \alpha \left( 2\| g '' \|_{L^{\infty}} - \frac{1}{A} \right) \right] \\
& <  g' (2  u (\overline{x},0 ) ),
\end{align*}
for $\alpha>\frac{2A\| g '' \|_{L^{\infty}}}{1- 2A\| g '' \|_{L^{\infty}}}$.
This proves the claim and, in turn, the lemma.
\end{proof}
We now show a property that implies the semiconvexity of $u^+ (\cdot, y)$, for any $y \in [0,A]$.
\begin{lemma} \label{semiconc2}
Let $u_{\scriptscriptstyle A}$ and $g$ be as in Theorem~\ref{opt_reg_theorem}, 
and let $u \in H^1 (\R^n \times (0,A))$ be a solution of \eqref{defcritptn}.
Then, there exists $\overline{D} > 0$ such that for every $(x, y) \in \R^n \times [0,A]$,
$$
\left[u(x+h,y)  + u(x-h,y) + \overline{D} |h|^2  \right]^+ \geq 2 u^+ (x, y) \qquad 
\forall\, h \in \R^n.
$$
In particular, for every $y \in [0,A]$ the function $u^+ (\cdot, y)$ is semiconvex, 
with semiconvexity constant~$\overline{D}$.
\end{lemma}
\noindent
An analogous result holds true for $u^-$.
\begin{lemma} \label{semiconc3}
Let $u_{\scriptscriptstyle A}$ and $g$ be as in Theorem~\ref{opt_reg_theorem}, 
and let $u \in H^1 (\R^n \times (0,A))$ be a solution of \eqref{defcritptn}. 
Then, there exists $\overline{C} > 0$ such that for every $(x, y) \in \R^n \times [0,A]$, 
$$
\left[u(x+h,y)  + u(x-h,y) - \overline{C} |h|^2  \right]^- \leq 2 u^- (x, y) \qquad 
\forall\, h \in \R^n.
$$
In particular, for every $y \in [0,A]$ the function $u^- (\cdot, y)$ is semiconcave, 
with semiconcavity constant~$\overline{C}$.
\end{lemma}
The following remark will be useful in the proof of Proposition~\ref{phaseseparation}.
\begin{remark} \label{useful phases sep}
Combining Lemmata~\ref{semiconc2} and~\ref{semiconc3}, we obtain
\begin{multline*}
\left[u(x+h,y)  + u(x-h,y) + \overline{D} |h|^2  \right]^+ \geq 2 u^+ (x, y)
\geq 2 u (x, y) \\
 \geq 2 u^- (x, y) 
\geq \left[u(x+h,y)  + u(x-h,y) - \overline{C} |h|^2  \right]^- ,
\end{multline*}
for every $(x, y) \in \R^n \times [0,A]$, and $h \in \R^n$.
\end{remark}

\begin{remark} \label{def constants}
Let $L_A$, $D_A$ and $C_A$ be given by Remark~\ref{DLuA}.
A careful inspection of the proof of Lemma~\ref{semiconc2} shows that 
one can choose
\[
\overline{D} = \frac{1}{c_A} 
\left[ D_A + \frac{4 B_{A,g}}{c_A^2} \right]
\quad \text{ and } \quad \overline{C} = \frac{1}{c_A} 
\left[ C_A + \frac{4 B_{A,g}}{c_A^2} \right],
\]
where 
\[
c_A:= 1 - 2 A \| g '' \|_{L^{\infty}},
\qquad B_{A,g}:=  A L_A \max \{ L_A \| g''' \|_{L^{\infty}} , 2 c_A c_g \| g'' \|_{L^{\infty}} \},
\]
and $c_g > 0$ 
is a positive constant such that 
\begin{equation} \label{def cg}
4 \frac{L_A}{c_A} \| g'' \|_{L^{\infty}} t 
+ D_A  \| g'' \|_{L^{\infty}} t^2
< g' (0^+) \quad \text{ for every } t \in [0,  1/c_g).
\end{equation}
\end{remark}
We only give the proof of Lemma \ref{semiconc2}, 
since that one of Lemma \ref{semiconc3} is analogous.
\begin{proof}[Proof of Lemma \ref{semiconc2}]
For every $h \in \mathbb{R}^n \setminus \{ 0 \}$, 
$\alpha >0$, $\varepsilon > 0$, and $C > 0$,
we define the function
$$
u^{h, \alpha, \varepsilon}_C (x,y) : = \left[\frac{u(x+h,y)  + u(x-h,y) + C|h|^2}{2} + \alpha \, C|h|^2 \left( 1 - \frac{y}{A} \right) \right]^+ 
\hspace{-.2cm}+ \varepsilon | h|^2,
$$
and set $C^{\alpha, \varepsilon}_h:= \inf \{ C > 0 : u^{h, \alpha, \varepsilon}_C > u^+ \text{ in }\R^n \times [0,A]\}$.
We claim that
\begin{equation} \label{C3}
C^{\alpha, \varepsilon}_h \leq \max\{D_A - 2 \varepsilon, f_{\varepsilon} (\alpha) \}
\quad \text{ for every } \, \, \alpha  > \frac{A\| g '' \|_{L^{\infty}}}{c_A} \, \,  \text{ and } \, \, 0 < \varepsilon < D_A/2,
\end{equation}
where 
$$
f_{\varepsilon} (\alpha) := 
2 \frac{B_{A,g} + \varepsilon c_A^2 A  \| g'' \|_{L^{\infty}}}
{ c_A^2 (\alpha c_A -  A \| g'' \|_{L^{\infty}} )},
$$
and the constants $c_A$, $B_{A, g}$, and $c_g$ are defined in Remark~\ref{def constants}.
%
Before proving the claim, let us show how this will imply the lemma.

Setting 
$$
G_{\varepsilon} (\alpha):= (1 + 2 \alpha)  \max\{D_A - 2 \varepsilon, f_{\varepsilon} (\alpha) \}, 
$$
from \eqref{C3} and by definition of $C^{\alpha, \varepsilon}_h $ it follows that
\begin{align}
&\left[u(x+h,y)  + u(x-h,y) + G_{\varepsilon} (\alpha)
|h|^2  \right]^+ 
\hspace{-.2cm}+ 2 \varepsilon | h|^2 \nonumber \\
&\geq\left[u(x+h,y)  + u(x-h,y) + C^{\alpha, \varepsilon}_h  (1 + 2 \alpha) 
|h|^2  \right]^+ 
\hspace{-.2cm}+ 2 \varepsilon | h|^2 \nonumber \\
&\geq \left[u(x+h,y)  + u(x-h,y) + C^{\alpha, \varepsilon}_h \left(1 + 2 \alpha \left( 1 - \frac{y}{A} \right) \right)  
|h|^2  \right]^+ 
\hspace{-.2cm}+ 2 \varepsilon | h|^2 \nonumber \\
&\geq 2 u^+ (x, y),  \label{chain}
\end{align}
for every $(x, y) \in \R^n \times [0, A]$, $\alpha > A\| g '' \|_{L^{\infty}} / c_A$, 
and $\varepsilon \in (0, D_A/2)$.
One can check that for every fixed $\varepsilon \in (0, D_A/2)$
$$
G_{\varepsilon} (\alpha) = 
\begin{cases}
(1 + 2 \alpha) f_{\varepsilon} (\alpha) &\text{ for } A\| g '' \|_{L^{\infty}} / c_A < \alpha < \alpha_{\varepsilon}, \\
(1 + 2 \alpha) (D_A - 2 \varepsilon) &\text{ for } \alpha \geq \alpha_{\varepsilon}, 
\end{cases}
$$
where 
$$
\alpha_{\varepsilon}:= \frac{A \| g'' \|_{L^{\infty}}}{c_A} 
+ 2 \frac{B_{A,g} + \varepsilon c_A^2 A  \| g'' \|_{L^{\infty}}}{c_A^3 (D_A - 2 \varepsilon)}.
$$
From this, it follows that for every $\varepsilon \in (0,  D_A/2)$
\begin{align*}
&\min \left\{ G_{\varepsilon} (\alpha) : \alpha > A\| g '' \|_{L^{\infty}} / c_A \right\} 
= G_{\varepsilon} (\alpha_{\varepsilon}) =  \overline{D} - 2 \varepsilon, 
\end{align*}
with $\overline{D}$ defined in Remark~\ref{def constants}.
Therefore, minimizing in $\alpha$ the left hand side of \eqref{chain} we obtain
\begin{align*}
\left[u(x+h,y)  + u(x-h,y) + ( \overline{D} - 2 \varepsilon ) |h|^2  \right]^+ 
\hspace{-.2cm}+ 2 \varepsilon | h|^2
\geq 2 u^+ (x, y), 
\end{align*}
for every $(x, y) \in \R^n \times [0, A]$, and $\varepsilon \in (0, D_A/2)$.
Taking the limit as $\varepsilon \to 0^+$ we conclude.

Let us now show \eqref{C3}.
By definition of $C^{\alpha, \varepsilon}_h$, the maximum principle, 
and thanks to \eqref{tozero}, there exists $(\overline{x}, \overline{y}) \in \R^n \times \{ 0, A \}$ such that
\begin{equation} \label{wq3}
0 = u^{h, \alpha}_{C^{\alpha, \varepsilon}_h} (\overline{x}, \overline{y}) - u^+ (\overline{x}, \overline{y})
= \inf_{\overline{\Omega}^+} ( u^{h, \alpha}_{C^{\alpha, \varepsilon}_h} - u^+ ).
\end{equation}
In the following we assume $C^{\alpha, \varepsilon}_h > 0$, since otherwise \eqref{C3} is trivially satisfied.
We have two possibilities.

\smallskip

\noindent
\textbf{Case 1: $\overline{y} = A$.} 
At the contact point $(\overline{x}, A)$ we have
%
\begin{align*}
u^+_A (\overline{x}) = \left[\frac{ u_{\scriptscriptstyle A} (\overline{x}+h)  + u_{\scriptscriptstyle A} (\overline{x}-h)+ C^{\alpha, \varepsilon}_h |h|^2}{2}  \right]^+ 
\hspace{-.2cm}+ \varepsilon |h|^2 >0,
\end{align*}
so that $u^+_A (\overline{x}) = u_{\scriptscriptstyle A} (\overline{x}) > 0$. Therefore
\begin{align*}
u_{\scriptscriptstyle A} (\overline{x}) &= \left[\frac{ u_{\scriptscriptstyle A} (\overline{x}+h)  + u_{\scriptscriptstyle A} (\overline{x}-h)+ C^{\alpha, \varepsilon}_h |h|^2}{2}  \right]^+ 
\hspace{-.2cm}+ \varepsilon |h|^2 \\
&\geq \frac{ u_{\scriptscriptstyle A} (\overline{x}+h)  + u_{\scriptscriptstyle A} (\overline{x}-h)+ C^{\alpha, \varepsilon}_h |h|^2}{2} + \varepsilon |h|^2 \\
&\geq \frac{ 2 u_{\scriptscriptstyle A} (\overline{x}) - D_A |h|^2 + C^{\alpha, \varepsilon}_h |h|^2}{2} + \varepsilon  |h|^2 
= u_{\scriptscriptstyle A} (\overline{x}) + \frac{1}{2} \left( C^{\alpha, \varepsilon}_h - D_A + 2 \varepsilon \right) |h|^2, 
\end{align*}
which implies 
\begin{equation} \label{condition 1a}
C^{\alpha, \varepsilon}_h \leq D_A - 2 \varepsilon.
\end{equation}

\smallskip

\noindent
\textbf{Case 2: $\overline{y} =0$.} 
At the contact point $(\overline{x},0)$ we have
\begin{equation} \label{contact point0} 
0< \left[ u(\overline{x} +h,0)  + u(\overline{x} -h,0) + C^{\alpha, \varepsilon}_h (1 + 2 \alpha )|h|^2  \right]^+ 
\hspace{-.2cm}+ 2 \varepsilon | h|^2 = 2 u^+ (\overline{x}, 0).
\end{equation}
Therefore, $u^+ (\overline{x}, 0) = u (\overline{x}, 0)$ and
\begin{align}
0 &< g' (2 u (\overline{x}, 0)) = \partial_y u (\overline{x}, 0) \leq (\partial_y u^{h, \alpha}_{C^{\alpha, \varepsilon}_h}) (\overline{x},0) \label{derivative contact} \\
&= \partial_y \left\{  \left[\frac{u(\overline{x}+h,y)  + u(\overline{x}-h,y) + C^{\alpha, \varepsilon}_h |h|^2}{2} + \alpha \, C^{\alpha, \varepsilon}_h 
|h|^2 \left( 1 - \frac{y}{A} \right) \right]^+ \right\}
\mid_{y = 0}. \nonumber
\end{align}
From the fact that the right hand side in the above expression is positive, it follows that
\begin{align*}
\frac{u(\overline{x}+h,y)  + u(\overline{x}-h,y) + C^{\alpha, \varepsilon}_h |h|^2}{2} + \alpha \, C^{\alpha, \varepsilon}_h |h|^2 \left( 1 - \frac{y}{A} \right) \geq 0 \quad \text{ for $y$ close to $0$},
\end{align*}
and from \eqref{derivative contact} we get
\begin{align}
0 <& g' (2 u (\overline{x}, 0)) 
\leq \frac{1}{2} \left[ \partial_y u (\overline{x} + h,0) + \partial_y u (\overline{x} - h,0) \right] 
- \frac{\alpha C^{\alpha, \varepsilon}_h |h|^2}{A}. \label{one2}
\end{align}
Moreover, identity \eqref{contact point0} becomes
\begin{equation} \label{contact point} 
 u(\overline{x} +h,0)  + u(\overline{x} -h,0) + C^{\alpha, \varepsilon}_h (1 + 2 \alpha )|h|^2  + 2 \varepsilon | h|^2 = 2 u (\overline{x}, 0), 
\end{equation}
Observing now that the role played by $u (\overline{x} + h,0)$ and $u (\overline{x} - h,0)$
is symmetric, we only need to consider three subcases.

\vspace{.2cm}

\noindent
\textbf{Case 2a: $\overline{y} =0$ with $u (\overline{x} + h,0) \geq 0$ and $u (\overline{x} - h,0) \geq 0$.} 
In this case, recalling \eqref{normalderiv}, from relation \eqref{one2} we obtain
\begin{align}
0 <& g' (2 u (\overline{x}, 0)) 
\leq \frac{1}{2} \left[ g' (2 u (\overline{x} + h,0) ) + g' (2 u (\overline{x} - h,0) ) \right] 
- \frac{\alpha C^{\alpha, \varepsilon}_h |h|^2}{A}. \label{one}
\end{align}
Let us now show that for every $a,b \geq 0$
\begin{align}
\frac{g' (a) + g' (b)}{2} \leq g' \left( \frac{a+b}{2} \right) + \frac{1}{8} \| g''' \|_{L^{\infty}} ( a-b )^2. \label{g'''}
\end{align}
Indeed, there exist $\theta, \tau \in (0,1)$ such that
\begin{align*}
g' (a) &= g' \left( \frac{a+b}{2} + \frac{a-b}{2} \right)
= g' \left( \frac{a+b}{2} \right) + g'' \left( \frac{a+b}{2} \right) \frac{a-b}{2} \\
&+ \frac{1}{2} g''' \mid_{ \frac{a+b}{2} + \theta \frac{a-b}{2}} \left(  \frac{a-b}{2} \right)^2,
\end{align*}
and
\begin{align*}
g' (b) &= g' \left( \frac{a+b}{2} - \frac{a-b}{2} \right)
= g' \left( \frac{a+b}{2} \right) - g'' \left( \frac{a+b}{2} \right) \frac{a-b}{2} \\
&\hspace{.2cm}+ \frac{1}{2} g''' \mid_{\frac{a+b}{2} - \tau \frac{a-b}{2}} \left(  \frac{a-b}{2} \right)^2.
\end{align*}
Summing up the last two relations we obtain the claim.
Applying \eqref{g'''} with 
$a = 2 u (\overline{x}+h,0)$ and 
$b = 2 u (\overline{x}-h,0)$, and using \eqref{contact point},
relation \eqref{one} gives
\begin{align*}
g' (2 u (\overline{x}, 0))  
& \leq g' \big(2 u (\overline{x},0) - (1 + 2 \alpha) C^{\alpha, \varepsilon}_h |h|^2 - 2 \varepsilon |h|^2 \big) \\
& + \frac{1}{2} \| g''' \|_{L^{\infty}} ( u (\overline{x}+h,0) - u (\overline{x}-h,0) )^2
 - \frac{\alpha C^{\alpha, \varepsilon}_h |h|^2}{A}.
\end{align*}
By Lemma \ref{Lip2} it follows that $u (\cdot, 0)$ is Lipschitz continuous,
with Lipschitz constant $L_A/c_A$. 
Therefore, recalling that $c_A=1-2A\| g'' \|_{L^{\infty}},$ we get
\begin{align*}
g' (2 u (\overline{x}, 0)) 
&\leq g' (2 u (\overline{x},0)) + \big[(1 + 2 \alpha) C^{\alpha, \varepsilon}_h + 2 \varepsilon \big] |h|^2 \| g'' \|_{L^{\infty}} 
+ 2 \frac{L^2_A}{c_A^2} \| g''' \|_{L^{\infty}} |h|^2  - \frac{\alpha C^{\alpha, \varepsilon}_h |h|^2}{A}  \\
& = g' (2 u (\overline{x},0)) + 
|h|^2 \left[2 \frac{L^2_A}{c_A^2} \| g''' \|_{L^{\infty}} + 2 \varepsilon \| g'' \|_{L^{\infty}} 
+   C^{\alpha, \varepsilon}_h\left(\| g'' \|_{L^{\infty}}  
- \alpha \frac{c_A}{A}  \right) \right]  \\
&= g' (2 u (\overline{x},0)) + 
|h|^2 \left[2 \frac{L^2_A}{c_A^2} \| g''' \|_{L^{\infty}} + 2 \varepsilon \| g'' \|_{L^{\infty}} 
-   C^{\alpha, \varepsilon}_h \frac{\alpha c_A - A \| g'' \|_{L^{\infty}}}{A}  \right]. 
\end{align*}
The inequality above is only possible if the last term in the right hand side 
is non-negative, that is, if 
\begin{equation} \label{condition 1b}
C^{\alpha, \varepsilon}_h \leq \frac{2 A ( L^2_A \| g''' \|_{L^{\infty}} + \varepsilon c_A^2 \| g'' \|_{L^{\infty}} )}
{ c_A^2 (\alpha c_A -  A \| g'' \|_{L^{\infty}} )}.
\end{equation}

\vspace{.2cm}

\noindent
\textbf{Case 2b: $\overline{y} =0$ with $u (\overline{x} + h,0) \geq 0$ and $u (\overline{x} - h,0) < 0$.}
In this case, recalling \eqref{normalderiv},  \eqref{one2} implies that
\begin{align}
0 < g' (2 u (\overline{x}, 0)) 
\leq \frac{1}{2} \left[ g' (2 u (\overline{x} + h,0) ) - g' (2 |u (\overline{x} - h,0)| ) \right] 
- \frac{\alpha C^{\alpha, \varepsilon}_h |h|^2}{A}. \label{different signs}
\end{align}
By Lemma \ref{Lip2}, $u (\cdot, 0)$ is Lipschitz continuous,
with Lipschitz constant $L_A/c_A$.
Therefore, the right hand side of the above expression can be estimated as follows
\begin{align*}
&\frac{1}{2} \left[ g' (2 | u (\overline{x} + h,0)| ) - g' (2 |u (\overline{x} - h,0)| ) \right] 
- \frac{\alpha C^{\alpha, \varepsilon}_h |h|^2}{A} \\
&\leq  
\| g'' \|_{L^{\infty}} \big| \,  |u (\overline{x} - h,0)| - | u (\overline{x} + h,0) | \, \big| 
- \frac{\alpha C^{\alpha, \varepsilon}_h |h|^2}{A} \\
&\leq  
2 \frac{L_A}{c_A} \| g'' \|_{L^{\infty}}  | h | - \frac{\alpha C^{\alpha, \varepsilon}_h |h|^2}{A}. 
\end{align*}
On the other hand, thanks to \eqref{contact point} 
we can estimate the left hand side of \eqref{different signs} as
\begin{align*}
&g' (2 u (\overline{x}, 0)) 
= g' \big( | u(\overline{x} +h,0)|  - | u(\overline{x} -h,0)| + C^{\alpha, \varepsilon}_h (1 + 2 \alpha )|h|^2  + 2 \varepsilon | h|^2 \big) \\
&\geq g' (0^+) - \| g'' \|_{L^{\infty}} \big|  | u(\overline{x} +h,0)|  - | u(\overline{x} -h,0)| \big|
- \| g'' \|_{L^{\infty}} [ C^{\alpha, \varepsilon}_h (1 + 2 \alpha )  + 2 \varepsilon ] |h|^2 \\
&\geq g' (0^+) - 2 \frac{L_A}{c_A} \| g'' \|_{L^{\infty}}  | h |
- \| g'' \|_{L^{\infty}} [ C^{\alpha, \varepsilon}_h (1 + 2 \alpha )  + 2 \varepsilon ] |h|^2.
\end{align*}
Combining the last two inequalities and \eqref{different signs} we obtain 
\begin{align}
g' (0^+) &\leq 4 \frac{L_A}{c_A} \| g'' \|_{L^{\infty}} | h |
+ \| g'' \|_{L^{\infty}} [ C^{\alpha, \varepsilon}_h (1 + 2 \alpha ) + 2 \varepsilon ]|h|^2
- \frac{\alpha C^{\alpha, \varepsilon}_h |h|^2}{A} \nonumber \\
&= 4 \frac{L_A}{c_A} \| g'' \|_{L^{\infty}} | h |
+  \left[ 2 \varepsilon  \| g'' \|_{L^{\infty}} 
-   C^{\alpha, \varepsilon}_h \frac{\alpha c_A - A \| g'' \|_{L^{\infty}}}{A} \right] |h|^2.
\label{last inequality of the list}
\end{align}
We now distinguish two subcases.

\vspace{.2cm}

\noindent
\textbf{Case 2bi: Small values of $|h|$.}
Let $c_g > 0$ be defined by \eqref{def cg}. 
From \eqref{last inequality of the list}
it follows that for every $| h | \in [0,  1/c_g)$ we have 
\begin{align*}
g' (0^+) &\leq 4 \frac{L_A}{c_A} \| g'' \|_{L^{\infty}} | h | 
+ 2 \varepsilon  \| g'' \|_{L^{\infty}} |h|^2 \\
&< 4 \frac{L_A}{c_A} \| g'' \|_{L^{\infty}} | h | 
+ D_A  \| g'' \|_{L^{\infty}} |h|^2 
< g' (0^+), 
\end{align*}
which is impossible. Therefore,  \eqref{last inequality of the list} 
can only be satisfied for $| h | \geq 1/c_g$.

\vspace{.2cm}

\noindent
\textbf{Case 2bii: $|h|$ large.}
Suppose now that $| h | \geq 1/c_g$, where $c_g$ is given by  \eqref{def cg}.
Then, $|h| \leq c_g |h|^2$ and thanks to \eqref{last inequality of the list}
we obtain
\begin{align*}
g' (0^+) &\leq  \left[ \left( 4 \frac{L_A c_g}{ c_A} 
+   2 \varepsilon \right)  \| g'' \|_{L^{\infty}} 
-   C^{\alpha, \varepsilon}_h \frac{\alpha c_A - A \| g'' \|_{L^{\infty}}}{A} \right] |h|^2.
\end{align*}
Last inequality is impossible, unless 
\begin{equation} \label{condition 1c}
C^{\alpha, \varepsilon}_h < \frac{\left( 4 \frac{L_A c_g}{c_A} 
+   2 \varepsilon \right) A \| g'' \|_{L^{\infty}}}{\alpha c_A - A \| g'' \|_{L^{\infty}}}
= \frac{2 A ( 2 L_A c_A c_g \| g'' \|_{L^{\infty}} + \varepsilon c_A^2 \| g'' \|_{L^{\infty}} )}
{ c_A^2 (\alpha c_A -  A \| g'' \|_{L^{\infty}} )}.
\end{equation}

\vspace{.2cm}

\noindent
\textbf{Case 2c: $\overline{y} =0$ with $u (\overline{x} + h,0) < 0$ and $u (\overline{x} - h,0) < 0$.}
In this case, inequality \eqref{one2} becomes
\begin{align*}
0 <& g' (2 u (\overline{x}, 0)) 
\leq - \frac{1}{2} \left[ g' (|u (\overline{x} + h,0)|)  + g' (|u (\overline{x} - h,0)|) \right] 
- \frac{\alpha C^{\alpha, \varepsilon}_h |h|^2}{A} < 0, 
\end{align*}
which is impossible.

\vspace{.2cm}

\noindent
\textbf{Case 2d: Proof of  \eqref{C3}.}
From the previous steps it follows that at least one among inequalities 
\eqref{condition 1a}, \eqref{condition 1b}, and \eqref{condition 1c}
has to be satisfied.
Recalling the definition of $f_{\varepsilon} (\alpha)$,
this concludes the proof of \eqref{C3} and, in turn, of the lemma.
\end{proof}
\end{section}

\section{Phases Separation and Optimal Regularity} \label{phases}

As already mentioned in the Introduction, the main problem in establishing optimal regularity
is that one cannot exclude \textit{a priori} the existence
of free boundary points where the function $u$ changes sign.
Indeed, at such points $\partial_y u (\cdot, 0)$ would be discontinuous, with a jump of $2 g' (0^+)$.
This is ruled out by the following proposition, which 
shows that the two ``phases'' $\{ x \in \mathbb{R}^n : u (x, 0) > 0 \}$ 
and $\{ x \in \mathbb{R}^n : u (x, 0) < 0 \}$ are well separated.
\begin{proposition} \label{phaseseparation}
Let $u_{\scriptscriptstyle A}$ and $g$ be as in Theorem~\ref{opt_reg_theorem}, 
let $u \in H^1 (\R^n \times (0,A))$ be a solution of \eqref{defcritptn},
and let $x \in \partial K_u$, where $K_u$ is defined by \eqref{fracture}. 
Then, there exists $r_0 =r_0 (x) \in (0, 1)$ such that
$$
B^n_{r_0} (x) \cap \overline{\{ x' \in \mathbb{R}^n : u (x', 0) > 0 \}}  
\cap \overline{\{ x' \in \mathbb{R}^n : u (x', 0) < 0 \}}  = \emptyset.
$$
\end{proposition}
Before proving Proposition~\ref{phaseseparation}, 
we show how this allows us to prove Theorem~\ref{opt_reg_theorem}.
\begin{proof}[Proof of Theorem~\ref{opt_reg_theorem}]
Let $x \in \partial K_u$. Without any loss of generality, thanks to Proposition~\ref{phaseseparation}, 
we can assume that 
$$
u (x', 0) \geq 0 \qquad \text{ for every } x' \in B^n_{r_0} (x), 
$$ 
where $r_0$ is given by Proposition~\ref{phaseseparation}. 
We claim that there exists $0 < \widehat{r} \leq r_0 $ and $D' > 0$, 
such that
\begin{equation} \label{semiconvexity u}
D^2_{xx} u (x', y) \geq  -D' \qquad \text{ for every } x' \in B_{\widehat{r}} (x, 0) \cap \{ y > 0 \}.
\end{equation}
Indeed, let us write $u = u_1 + u_2 + u_3$, where $u_1$, $u_2$, and $u_3$ are the harmonic functions 
in $\R^{n} \times (0,A)$ with the following boundary conditions:
$$
\begin{cases}
u_1 = 0 &\text{ on } \{ y = 0 \}, \\
u_1 = u_{\scriptscriptstyle A} &\text{ on } \{ y = A \}, 
\end{cases}
\qquad 
\begin{cases}
u_2 = u^+ &\text{ on } \{ y = 0 \}, \\
u_2 = 0 &\text{ on } \{ y = A \}, 
\end{cases}
\qquad 
\begin{cases}
u_3 = u^- &\text{ on } \{ y = 0 \}, \\
u_3 = 0 &\text{ on } \{ y = A \}. 
\end{cases}
$$ 
Note now that $u_3$ is $C^{\infty}$ in a neighborhood of $(x, 0)$, 
since $u^- = 0$ in $B^n_{r_0} (x)$.
Analogously, $u_1$ is also $C^{\infty}$ in a neighborhood of $(x, 0)$.
On the other hand, by maximum principle $u_2 \geq 0$. 
Therefore, an argument similar to the one used in the proof of 
Lemma~\ref{semiconc2} shows that, for every $y \in [0,A]$, 
$u_2 (\cdot, y)$ is semiconvex.
Therefore, 
$$
D^2_{xx} u_2 (x', y) \geq - \overline{D} \quad \text{ for every } (x', y) \in \R^n \times (0, A).
$$ 
Then, using the fact that $u_1$ and $u_3$ are smooth, \eqref{semiconvexity u} follows.

We now note that $v$ defined in \eqref{v intro} 
is a harmonic function in $\R^n \times (0,A)$ satisfying
$$
v\geq 0\quad \text{and}\quad v [ \partial_y v + g' (0^+) - g' (2 |v|)] \equiv 0 \qquad \text{on }\{y=0\} \cap B^n_{\widehat{r}} (x), 
$$
which is just a minor variation of the classical Signorini problem $v\partial_yv=0$ \cite{AC,ACS}.
Thus, the remaining part of the proof of Theorem~\ref{opt_reg_theorem} 
can easily be obtained by repeating (with the needed minor modifications) the arguments used in \cite{AC, CF}.
\end{proof}
We now give the proof of Proposition~\ref{phaseseparation}.
\begin{proof}[Proof of Proposition~\ref{phaseseparation}]
Without any loss of generality, we can assume $x = 0$.
We will argue by contradiction, assuming that
\begin{equation} \label{contradic}
B^n_{r} \cap \overline{\{ x' \in \mathbb{R}^n : u (x', 0) > 0 \}}  
\cap 
\overline{\{ x' \in \mathbb{R}^n : u (x', 0) < 0 \}} \neq \emptyset \qquad \text{ for every } r > 0.
\end{equation}
We divide the proof into two steps.

\vspace{.2cm}

\noindent
\textbf{Step 1.} We show that \eqref{contradic} implies  
that $u (\cdot, 0)$ is differentiable at $x=0$ with $\nabla_x u (0,0) = 0$. 
Since $u (0, 0) = 0$, $u^+\geq 0$, and $u^-\leq 0$, we have 
$$
0 \in \partial^-_x u^+ (0, 0) \quad \text{ and } \quad 0 \in \partial^+_x u^- (0,0),
$$
where we denote by $\partial^-_x u^+ (\cdot, 0)$ and $\partial^+_x u^- (\cdot,0)$ 
the subdifferential of $u^+ (\cdot, 0)$  
and the superdifferential of $u^- (\cdot, 0)$, respectively.
Suppose now that \eqref{contradic} is satisfied but, by contradiction, there exists $\xi \in \R^n$ such that
$$
\xi \in \big( \partial^-_x u^+ (0, 0) \cup \partial^+_x u^- (0,0) \big) \setminus \{ 0 \}.
$$
Without loss of generality, we can assume $\xi \in \partial^-_x u^+ (0, 0)$ and 
$\xi = b \mathbf{e}_1$  for some $b > 0$, i.e.
\begin{equation} \label{xi}
b \mathbf{e}_1 \in \partial^-_x u^+ (0, 0), \quad \text{ for some } b > 0.
\end{equation}
Since, by Lemma~\ref{semiconc2}, $u^+ (\cdot, 0)$ is semiconvex with semiconvexity constant $\overline{D}$, 
\eqref{xi} implies that
\begin{equation} \label{Bxb}
u^+ (x, 0) + \overline{D} |x|^2 \geq u^+ (0, 0) + b \mathbf{e}_1 \cdot x \quad \text{ for every } x\in \R^n.
\end{equation}
Setting $x_b : = \frac{b}{2 \overline{D}} \mathbf{e}_1$, the above inequality can be written as
$u^+ (x, 0) \geq \overline{D} ( |x_b|^2 - | x - x_b|^2 )$, so that
\begin{equation} \label{xb}
u (x, 0) > 0 \quad \text{ for every } x \in B^n_{|x_b|} (x_b).
\end{equation}
We now divide the proof of Step 1 into two substeps.

\vspace{.2cm}

\noindent
\textbf{Step 1a.} We show that
$$
\eqref{xi} \quad \Longrightarrow \quad \partial^+_x u^- (0, 0) \setminus \{ 0 \} \neq \emptyset.
$$
Suppose, by contradiction, that \eqref{xi} is satisfied but $\partial^+_x u^- (0, 0) = \{ 0 \}$. 
Then, since $u^-$ is semiconcave, $u^- (\cdot, 0)$ is differentiable in $0$ and 
\begin{equation} \label{u-diff}
u (x, 0) \geq u^- (x, 0) \geq o (|x|) \quad \text{ for every } x \in \R^n. 
\end{equation}
By \eqref{contradic}, we can find a sequence $\{ x_k \}_{k \in \mathbb{N}}
\subset \R^n \setminus \{ 0 \}$ with $x_k \to 0$ such that
\begin{equation} \label{u neg k}
u (x_k, 0) < 0 \quad \text{ for every } k \in \mathbb{N}.
\end{equation}
Setting $h_k:= 2 |x_k| \mathbf{e}_1$, thanks to \eqref{u-diff} we have 
\begin{equation} \label{u-k}
u (x_k - h_k, 0 ) \geq u^- (x_k - h_k, 0 ) = o (|x_k - h_k|) = o (|x_k|), 
\end{equation}
where the last equality follows from our choice of the sequence $\{ h_k \}_{k \in \mathbb{N}}$.
On the other hand, by \eqref{Bxb} it follows that
\begin{align} 
&u^+ (x_k + h_k, 0)  \geq b \mathbf{e}_1 \cdot (x_k + h_k) - \overline{D} |x_k + h_k|^2 \nonumber \\
&= b |x_k| \left( 2 + \mathbf{e}_1 \cdot \frac{x_k}{|x_k|}  \right)
- \overline{D} |x_k|^2 \left( 5 + 4 \mathbf{e}_1 \cdot \frac{x_k}{|x_k|} \right) \nonumber \\
& \geq b |x_k| 
- 9\overline{D} |x_k|^2  \geq \frac{b}{2} |x_k|,  \label{u+k}
\end{align}
for $k$ sufficiently large.
Thanks to Remark~\ref{useful phases sep}, combining \eqref{u neg k}, 
\eqref{u-k}, and \eqref{u+k}, we have 
that, for $k$ large enough,
\begin{align*}
0 &> 2 u (x_k, 0 ) \geq 2 u^- (x_k, 0 ) \\
&\geq \left[u (x_k + h_k, 0 )  + u (x_k - h_k, 0 ) - \overline{C} |h_k|^2  \right]^- \\
& \geq \left[ \frac{b}{2} |x_k|  + o (|x_k|) - 4 \overline{C} |x_k|^2  \right]^- = 0, 
\end{align*}
which is impossible.

\vspace{.2cm}

\noindent
\textbf{Step 1b.} We conclude the proof of Step 1.
By Step 1a, there exists $d > 0$ and $\mathbf{e} \in \mathbb{S}^{n} \cap  \{ y = 0 \}$ 
such that $d \mathbf{e} \in \partial^+_x u^- (0,0)$.
Since, by Lemma~\ref{semiconc3}, $u^- (\cdot, 0)$ 
is semiconcave with semiconcavity constant $\overline{C}$, by repeating the same argument used 
to show \eqref{Bxb} we have that
\begin{equation}
\label{xd}
u (x, 0) < 0 \qquad \text{ for every } x \in B^n_{|x_{d}|} (x_{d}), 
\end{equation}
where we set $x_d : = \frac{d}{2 \overline{C}} \mathbf{e}$.
Taking into account \eqref{xb}, this implies $\mathbf{e} = - \mathbf{e}_1$, 
thus $x_d = - \frac{d}{2 \overline{C}} \mathbf{e}_1$.
We will now show that $\partial_{x_1} u (\cdot, 0)$ is unbounded, 
against Lemma~\ref{Lip2}. 

To this aim, for every $\varepsilon > 0$ we set $w_{\varepsilon}:= - \varepsilon \mathbf{e}_1$.
In this way, $w_\varepsilon \to 0$ as $\varepsilon \to 0^+$ and 
$w_{\varepsilon}  \in B^n_{|x_{d}|} (x_{d})$
for $\varepsilon$ sufficiently small, so that $u (w_\varepsilon ,0 ) < 0$.
We claim that 
\begin{equation} \label{unbounded gradient}
\lim_{\varepsilon \to 0^+} \partial_{x_1} u (w_{\varepsilon},0) = + \infty.
\end{equation}
Let $\tilde{u}: \R^n \times[0, \infty) \to \R$ be the harmonic extension
of $u (x, 0)$ to the half space $\R^n \times[0, \infty)$. 
We have 
\begin{align*}
&\frac{1}{C_n} \partial_{x_1} u (w_{\varepsilon},0) 
= \frac{1}{C_n} \partial_{x_1} \tilde{u} (w_{\varepsilon},0) 
= \int_{\mathbb{R}^n} \frac{((w_{\varepsilon})_1 -  z_1)}
{|w_{\varepsilon} - z|^{n+1}} 
\left( \partial_y \tilde{u} (w_{\varepsilon}, 0) - \partial_y \tilde{u} (z, 0) \right) \, d z \\
&= \int_{\mathbb{R}^n} \frac{((w_{\varepsilon})_1 -  z_1)}
{|w_{\varepsilon} - z|^{n+1}} 
\left( \partial_y (\tilde{u} -u) (w_{\varepsilon}, 0) - \partial_y (\tilde{u} -u)(z, 0) \right) \, d z \\
&+ \int_{\mathbb{R}^n} \frac{((w_{\varepsilon})_1 -  z_1)}
{|w_{\varepsilon} - z|^{n+1}} 
\left( \partial_y u (w_{\varepsilon}, 0) - \partial_y u (z, 0) \right) \, d z,
\end{align*}
for some positive dimensional constant $C_n$.
Since $\tilde{u} - u$ vanishes on $\{ y = 0 \}$ 
and is harmonic in $\R^n \times (0,A)$, 
we have $x \mapsto \partial_y (\tilde{u} -u) (x, 0) \in C^{\infty} (\R^n)$. 
Therefore, to prove our claim it will be sufficient to show that 
the last integral diverges as $\varepsilon \to 0^+$.
Let $f : \R^n \to \mathbb{R}$ be defined as
$$
f (x) : = 
\begin{cases}
\partial_y u (x, 0) & \text{ if } u (x, 0) < 0, \\
- g' (0^+) & \text{ if } u (x, 0) \geq 0.
\end{cases}
$$
Since $\partial_y u (x, 0)=-g'(2|u (x, 0)|)$ where $u (x, 0) < 0$,
it follows that
that $f$ is Lipschitz continuous, with Lipschitz constant 
$2 \| g '' \|_{L^{\infty}} \frac{L_A}{1-2A\| g '' \|_{L^{\infty}}}$ (recall Lemma~\ref{Lip2}).
In the following, we set 
\begin{equation} \label{def lambda}
\lambda:= \min \left\{ \frac{b}{\overline{D}} , \frac{d}{\overline{C}} \right\}.
\end{equation}
Given $r > 0$ with $0< r < \lambda / 4$, we split the integral 
under consideration as follows: 
\begin{align*}
&\int_{\mathbb{R}^n} \frac{((w_{\varepsilon})_1 -  z_1)}
{|w_{\varepsilon} - z|^{n+1}} 
\left( \partial_y u (w_{\varepsilon}, 0) - \partial_y u (z, 0) \right) \, d z \\
&= \int_{\mathbb{R}^n \setminus B^n_r} \frac{((w_{\varepsilon})_1 -  z_1)}
{|w_{\varepsilon} - z|^{n+1}} 
\left( \partial_y u (w_{\varepsilon}, 0) - \partial_y u (z, 0) \right) \, d z \\
&+ \int_{B^n_r} \frac{((w_{\varepsilon})_1 -  z_1)}
{|w_{\varepsilon} - z|^{n+1}} 
\left( \partial_y u (w_{\varepsilon}, 0) - \partial_y u (z, 0) \right) \, d z,
\end{align*}
We can disregard the first integral, which is bounded for $\varepsilon$ small enough.
Concerning the second integral, using the fact that $u (w_{\varepsilon}, 0) < 0$, we have
\begin{align*}
& \int_{B^n_r} \frac{((w_{\varepsilon})_1 -  z_1)}
{|w_{\varepsilon} - z|^{n+1}} 
\left( \partial_y u (w_{\varepsilon}, 0) - \partial_y u (z, 0) \right) \, d z
= \int_{B^n_r} \frac{((w_{\varepsilon})_1 -  z_1)}
{|w_{\varepsilon} - z|^{n+1}} 
\left( \partial_y u (w_{\varepsilon}, 0) - f (w_{\varepsilon}) \right) \, d z \\
&+ \int_{B^n_r} \frac{((w_{\varepsilon})_1 -  z_1)}
{|w_{\varepsilon} - z|^{n+1}} 
\left( f (z) - \partial_y u (z, 0) \right) \, d z
+ \int_{B^n_r} \frac{((w_{\varepsilon})_1 -  z_1)}
{|w_{\varepsilon} - z|^{n+1}} 
\left( f (w_{\varepsilon}) - f (z) \right) \, d z \\
&= \int_{B^n_r} \frac{((w_{\varepsilon})_1 -  z_1)}
{|w_{\varepsilon} - z|^{n+1}} 
\left( f (z) - \partial_y u (z, 0) \right) \, d z
+ \int_{B^n_r} \frac{((w_{\varepsilon})_1 -  z_1)}
{|w_{\varepsilon} - z|^{n+1}} 
\left( f (w_{\varepsilon}) - f (z) \right) \, d z \\
&=: I_1^{\varepsilon} + I_2^{\varepsilon}.
\end{align*}
Since $f$ is Lipschitz continuous, $I_2^{\varepsilon}$ is uniformly bounded in $\varepsilon$. 
By definition of $f$ and by \eqref{xb}, we can split  $I_1^{\varepsilon}$ in the following way:
\begin{align*}
&I_1^{\varepsilon} =\int_{B^n_r} \frac{((w_{\varepsilon})_1 -  z_1)}
{|w_{\varepsilon} - z|^{n+1}} 
\left( f (z) - \partial_y u (z, 0) \right) \, d z 
= \int_{B^n_r \cap \{ u \geq 0 \}} \frac{((w_{\varepsilon})_1 -  z_1)}
{|w_{\varepsilon} - z|^{n+1}} 
\left( f (z) - \partial_y u (z, 0) \right) \, d z \\
&= \int_{B^n_r \cap B^n_{|x_b|} (x_b)} \frac{((w_{\varepsilon})_1 -  z_1)}
{|w_{\varepsilon} - z|^{n+1}} \left( f (z) - \partial_y u (z, 0) \right) \, d z \\ 
&+ \int_{B^n_r \cap \{ u \geq 0 \} \setminus B^n_{|x_b|} (x_b)} \frac{((w_{\varepsilon})_1 -  z_1)}
{|w_{\varepsilon} - z|^{n+1}}
\left( f (z) - \partial_y u (z, 0) \right) \, d z = : I_{1,1}^{\varepsilon} + I_{1,2}^{\varepsilon}.
\end{align*}
We claim that $I_{1,2}^{\varepsilon}$ is uniformly bounded in $\varepsilon$.
Indeed, first of all we observe that, thanks to 
\eqref{normalderiv} and Lemma \ref{Lip2},
\begin{align*}
&| f (z) - \partial_y u (z, 0) | 
= g' (0^+) +  \partial_y u (z, 0)  \leq  
g' (0^+) +  g' (2 u (z, 0)) \\
&\leq 2 \big( \, g' (0^+) + \| g '' \|_{L^{\infty}} u (z, 0) \, \big)
\leq 2 \left( g' (0^+) + r \frac{L_A \| g '' \|_{L^{\infty}}}{1-2A\| g '' \|_{L^{\infty}}} \right)
 =: c_r,
\end{align*}
for every $z \in B^n_r \cap \{ u \geq 0 \}$. Then, recalling \eqref{xd}, we have
\begin{align*}
| I_{1,2}^{\varepsilon} | &\leq c_r  
\int_{B^n_r \cap \{ u \geq 0 \} \setminus B^n_{|x_b|} (x_b)} \frac{1}
{|w_{\varepsilon} - z|^{n}} \, d z
= c_r   \int_{B^n_r \setminus  ( B^n_{|x_b|} (x_b) \cup B^n_{|x_d|} (x_d)) } 
\frac{1} {|w_{\varepsilon} - z|^{n}} \, d z \\
&= c_r  \int_0^r \int_{\mathbb{S}^{n-1} 
\cap \{ \omega \in \mathbb{S}^{n-1} \colon - \rho \, \overline{C} / d < \omega_1 < \rho \overline{D} / b \}}
\frac{\rho^{n-1}}{|w_{\varepsilon} - \rho \omega |^{n}} \, d \mathcal{H}^{n-1} (\omega) \, d \rho \\
&\leq c_r  \int_0^r \int_{\Sigma_{\rho}}
\frac{\rho^{n-1}}{|w_{\varepsilon} - \rho \omega |^{n}} \, d \mathcal{H}^{n-1} (\omega) \, d \rho,
\end{align*}
where we set
$$
\Sigma_{\rho} : = \left\{ \omega \in \mathbb{S}^{n-1} : |\omega_1| < \frac{\rho}{\lambda} \right \},
$$
and $\lambda$ is defined by \eqref{def lambda}.
Since $w_{\varepsilon}:= - \varepsilon \mathbf{e}_1$, we note that for every $\rho \in (0,r)$ and $\omega \in \Sigma_{\rho}$ 
$$
|w_{\varepsilon} - \rho \, \omega |^2 =  \rho^2  + 2 \varepsilon \rho \, \omega_1 + \varepsilon^2 
> \rho^2  - 2 \varepsilon \rho^2 \, \lambda+ \varepsilon^2
= \rho^2  \left( 1 -  2 \varepsilon \lambda \right) + \varepsilon^2 > \frac{1}{2} \rho^2,
$$
for $\varepsilon$ sufficiently small.
Therefore, for $\varepsilon$ sufficiently small we obtain
\begin{align*}
| I_{1,2}^{\varepsilon} | 
&\leq 2^{\frac{n}{2}} c_r  \int_0^r \int_{\Sigma_{\rho}}
\frac{1}{\rho} \, d \mathcal{H}^{n-1} (\omega) \, d \rho
\leq 2^{\frac{n}{2}} c_r \,  C,
\end{align*}
where we used the fact that $\mathcal{H}^{n-1} (\Sigma_{\rho}) \leq C \rho$ 
for some positive constant $C = C (n)$.

Let us now estimate $I_1^{\varepsilon}$.
Since $z_1 > 0$ for every $z \in B^n_r \cap B^n_{|x_b|} (x_b)$ 
and $(w_{\varepsilon})_1 = - \varepsilon < 0$, we have 
$(w_{\varepsilon})_1 -  z_1 < 0$. Therefore, since $g'>0$,
\begin{align*}
I_{1,1}^{\varepsilon}& = \int_{B^n_r \cap B^n_{|x_b|} (x_b)} \frac{(w_{\varepsilon})_1 -  z_1}
{|w_{\varepsilon} - z|^{n+1}} \left( f (z) - \partial_y u (z, 0) \right) \, d z \\
&= \int_{B^n_r \cap B^n_{|x_b|} (x_b)} \frac{z_1 - (w_{\varepsilon})_1 }
{|w_{\varepsilon} - z|^{n+1}} \left( g' (0^+) + g' (2 u (z, 0) \right) \, d z \\
&\geq g' (0^+) \int_{B^n_r \cap B^n_{|x_b|} (x_b)} \frac{z_1 - (w_{\varepsilon})_1 }
{|w_{\varepsilon} - z|^{n+1}} \, d z 
\geq g' (0^+) \int_{B^n_{r- \varepsilon} (w_{\varepsilon}) \cap B^n_{|x_b|} (x_b)} \frac{z_1 - (w_{\varepsilon})_1 }
{|w_{\varepsilon} - z|^{n+1}} \, d z \\
&= g' (0^+) \int_{B^n_{r- \varepsilon} \cap B^n_{|x_b|} (x_b - w_{\varepsilon})} \frac{\tau_1}
{|\tau|^{n+1}} \, d \tau 
= g' (0^+) \int_{\varepsilon}^{r - \varepsilon} \frac{1}{\rho} 
\int_{\Sigma^{\varepsilon}_{\rho}} \omega_1 \, d \mathcal{H}^{n-1} (\omega) \, d \rho,
\end{align*}
where 
$$
\Sigma^{\varepsilon}_{\rho} := \left\{ \omega \in \mathbb{S}^{n-1}: 
\frac{\overline{D} (\rho^2 + \varepsilon^2) + \varepsilon b}{\rho (b + 2 \varepsilon \overline{D})}
< \omega_1 < 1  
\right\}.
$$
Note now that, for $\varepsilon$ sufficiently small, 
since $\rho < r - \varepsilon < \lambda / 4 \leq b/(4\overline{D})$, we have 
$$
\frac{\overline{D} (\rho^2 + \varepsilon^2) + \varepsilon b}{\rho (b + 2 \varepsilon \overline{D})} 
< 2 \frac{\overline{D} \rho}{b} < \frac{1}{2}.
$$
Therefore, 
\begin{align*}
I_{1,1}^{\varepsilon}& \geq g' (0^+) \int_{\varepsilon}^{r - \varepsilon} \frac{1}{\rho} 
\int_{\mathbb{S}^{n-1} \cap \{ 1/2 < \omega_1 < 1 \}} \omega_1 \, d \mathcal{H}^{n-1} (\omega) \, d \rho \\
&= c_n \, g' (0^+) \int_{\varepsilon}^{r - \varepsilon} \frac{1}{\rho} \, d \rho 
= c_n \, g' (0^+) \ln \left( \frac{r - \varepsilon}{\varepsilon} \right),
\end{align*}
for some positive dimensional constant $c_n$.
Taking the limit as $\varepsilon \to 0^+$ we obtain
$$
\lim_{\varepsilon \to 0^+} I_{1,1}^{\varepsilon} = + \infty, 
$$
which proves \eqref{unbounded gradient}.
As noted before, this contradicts Lemma \ref{Lip2}, concluding the proof of Step~1.

\vspace{.2cm}

\noindent
\textbf{Step 2.} We show that there exist positive constants $\gamma$, $\eta,$ and $\overline{r}$ such that
\begin{equation} \label{soon a contradiction}
\partial_y u \geq \frac{3}{4} g' (0^+) 
\qquad \text{ in } R_{\gamma \overline{r}} 
:= B_{\gamma \overline{r}}^n \times \left[ (1 - \gamma ) \eta \overline{r} , (1 + \gamma ) \eta \overline{r} \right].
\end{equation}
By Step 1 we know that $u (\cdot, 0)$ is differentiable at $x=0$ with $\nabla_x u (0,0) = 0$,
hence there exists a continuous function $\sigma: [0, \infty) \to [0, \infty)$ 
with $\sigma (0) = 0$ such that 
$$
| u (x , 0) | \leq \sigma (|x|) |x| \qquad \text{ for every } x \in \R^n.
$$
Note that, with no loss of generality, we can assume that $\sigma(r)\geq r$ for all $r$.

Let $M, C, \eta$ be positive constants to be chosen later, and for every $r \in (0,1)$ set
$$
\Gamma_r := B_r^n \times \left[ 0, \eta r \right].
$$
We consider the harmonic function $V^+: \Gamma_r \to \mathbb{R}$ defined as
$$
V^+ (x, y):= u (x, y) - M \frac{\sigma (r)}{r} |x - x^+_0|^2 + n M \frac{\sigma (r)}{r} y^2 - (g' (0^+) - C r) y,
$$
where $x^+_0 \in \mathbb{R}^n$ is such that $u (x^+_0 , 0) > 0$ and $|x^+_0| = c_0 r$
with $0 < c_0 \ll 1$ (note that that such a point $x^+_0$ exists, because we are assuming, by contradiction,
that $(0,0)$ is a boundary point both for $\{ u > 0\}$ and for $\{ u < 0\}$).
Since $V^+$ is harmonic, we have 
$$
0 < \max_{\overline{\Gamma}_r} V^+ = \max_{\partial \Gamma_r } V^+, 
$$
where the positivity comes from the fact that $V^+ (x^+_0,0) = u (x^+_0,0) > 0$.
We now have several possibilities.
\vspace{.2cm}

\noindent
\textbf{Case 2a:} We show that, if $C$ is sufficiently large, then 
there exists no $\overline{x} \in \partial \Gamma_r \cap \{ y = 0\}$ such that
$$
\max_{\overline{\Gamma}_r} V^+ = V^+ (\overline{x}, 0).
$$
Suppose that such $\overline{x}$ exists.
Then, it cannot be $u (\overline{x}, 0) \leq 0$, since in that case 
$$
V^+ (\overline{x}, 0) = u (\overline{x}, 0) - M \frac{\sigma (r)}{r} | \overline{x} - x^+_0|^2 \leq 0.
$$
Therefore, $u (\overline{x}, 0) > 0$, and 
\begin{align*}
\partial_y V^+ (\overline{x}, 0) &= \partial_y u (\overline{x}, 0) -  g' (0^+) + C r
= g' (2 u (\overline{x}, 0)) - g' (0^+) + C r \\
&\geq - 2 u (\overline{x}, 0) \| g'' \|_{L^{\infty}} + C r 
\geq - 2 \sigma (r) \, r \| g'' \|_{L^{\infty}} + C r. 
\end{align*}
If we choose $C$ large enough we obtain $\partial_y V^+ (\overline{x}, 0) > 0$, which is impossible.

\vspace{.2cm}

\noindent
\textbf{Case 2b:} We show that, if $M$ is sufficiently large and $\eta$ is sufficiently small, then 
there exists no point $(\overline{x}, \overline{y})$ with $|\overline{x}| = r$
and $\overline{y} \in [0, \eta r]$ such that
$$
\max_{\overline{\Gamma}_r} V^+ = V^+ (\overline{x}, \overline{y}).
$$
Indeed, suppose that such a point $(\overline{x}, \overline{y})$ exists. 
Then, since $|\overline{x}| = r$ and $|x_0^+ | = c_0 r$, we have
$$
(1 - c_0)^2 r^2 \leq |\overline{x} - x^+_0|^2 \leq (1 + c_0)^2 r^2.
$$
Thus,

\begin{align*}
0 &< V^+ (\overline{x}, \overline{y}) 
= u (\overline{x}, \overline{y}) - M \frac{\sigma (r)}{r} |\overline{x} - x^+_0|^2 
+ n M \frac{\sigma (r)}{r} \overline{y}^2 - (g' (0^+) - C r) \overline{y} \\
&\leq u  (\overline{x}, \overline{y}) - M \,  (1 - c_0)^2 \sigma (r) r
+ n M \frac{\sigma (r)}{r} \overline{y}^2 - (g' (0^+) - C r) \overline{y}, 
\end{align*}
so that (recall that $c_0\ll 1$)
\begin{align}
u  (\overline{x}, \overline{y}) &> 
M \,  (1 - c_0)^2 \sigma (r) r - n M \frac{\sigma (r)}{r} \overline{y}^2 + (g' (0^+) - C r) \overline{y} \nonumber \\
&\geq M \,  (1 - c_0)^2 \sigma (r) \, r - ( n M \, \eta^2 \, \sigma (r) \, r + C \eta r^2 ) + g' (0^+) \overline{y} \nonumber \\
&\geq \frac{M}{2} \sigma (r) \, r + g' (0^+) \overline{y}, \nonumber 
\end{align}
for $\eta$ small enough.
Thanks to \eqref{bound partial_y}, this last estimate gives
$$
\frac{M}{2} \sigma (r) \, r + g' (0^+) \overline{y} 
\leq u  (\overline{x}, \overline{y}) 
\leq u  (\overline{x}, 0)
+ g' (0^+) \, \overline{y} + C_1 \, \frac{\overline{y}^2}{2} 
\leq \sigma (r) \, r + C_1 \, \eta^2 \frac{r^2}{2} 
+ g' (0^+) \, \overline{y}, 
$$
which is impossible for $M$ sufficiently large.

\vspace{.2cm}

\noindent
\textbf{Case 2c:} We show that, if $M$ is sufficiently large and $\eta$ is sufficiently small, then 
there exists no point $\overline{x} \in \mathbb{R}^n$ with $r/2 \leq |\overline{x}| \leq r$ such that
$$
\max_{\overline{\Gamma}_r} V^+ = V^+ (\overline{x}, \eta r).
$$
This case can be treated as Case 2b.

\vspace{.2cm}

\noindent
\textbf{Case 2d:} We show that, if $M$ is sufficiently large and $\eta$ is sufficiently small, 
there exist $\gamma, \overline{r} > 0$ such that \eqref{soon a contradiction} is satisfied.
From Cases 2a--2c, there exists $\overline{x} \in \mathbb{R}^n$ with $|\overline{x}| < r/2$ such that
$$
\max_{\overline{\Gamma}_r} V^+ = V^+ (\overline{x}, \eta r),
$$
so that
\begin{align*}
0 &\leq \partial_y V^+ (\overline{x}, \eta r) = 
\partial_y u (\overline{x}, \eta r) +2 n M  \eta \, \sigma (r) - (g' (0^+) - C r).
\end{align*}
Using the fact that $r \leq \sigma (r)$ we have
\begin{equation} \label{estimatedu}
\partial_y u (\overline{x}, \eta r)  \geq  
g' (0^+) - 2 n M  \eta \, \sigma (r) - C r \geq g' (0^+) - C_{\eta} \, \sigma ( r ),
\end{equation}
where $C_{\eta} = 2 n M \eta  + C$.
For $\gamma \in (1/2,1)$ let us set
$$
(\overline{x}, \eta r)  \in R_{\gamma r}:= B_{\gamma r}^n \times \left[ (1 - \gamma ) \eta r , (1 + \gamma ) \eta r \right]. 
$$
Thanks to \eqref{bound partial_y}, the function 
$g' (0^+) + 2 C_1 \eta r - \partial_y u$ is harmonic and nonnegative in $\Gamma_{2 r}$.
Thus, by \eqref{estimatedu} and Harnack inequality, there exists a constant $C_{\gamma} > 0$ such that
\begin{multline*}
\sup_{R_{\gamma r} } \left(  g' (0^+) + 2 C_1 \eta r - \partial_y u \right)
\leq C_{\gamma} \inf_{R_{\gamma r}} \left(  g' (0^+) + 2 C_1 \eta r - \partial_y u \right) \\
\leq  C_{\gamma} \left(  g' (0^+) + 2 C_1 \eta r  - (g' (0^+) - C_{\eta} \, \sigma ( r ))   \right)
= C_{\gamma} ( 2 C_1 \eta r + C_{\eta} \, \sigma ( r ) ).
\end{multline*}
From the previous chain of inequalities we obtain
$$
\partial_y u 
\geq g' (0^+) + 2 C_1 (1 - C_{\gamma}) \eta r  - C_{\gamma} C_{\eta} \, \sigma ( r )
\qquad \text{ in } R_{\gamma r}.
$$
Therefore, there exists $\overline{r} = \overline{r}  (\gamma)$ such that
$$
\partial_y u \geq \frac{3}{4} g' (0^+) 
\qquad \text{ in } R_{\gamma \overline{r}},
$$ 
thus showing \eqref{soon a contradiction}.
\vspace{.2cm}

\noindent
\textbf{Step 3.} We conclude. An argument analogous to that one 
used in Step 2 can be applied to the harmonic function 
$V^-: \Gamma_r \to \mathbb{R}$ defined as
$$
V^- (x, y):= u (x, y) + M \frac{\sigma (r)}{r} |x - x^-_0|^2 - n M \frac{\sigma (r)}{r} y^2 + (g' (0^+) - C r) y,
$$
where $x^-_0 \in \mathbb{R}^n$ is such that $u (x^-_0 , 0) < 0$ 
and $|x^-_0| = c^-_0 r$ with $0 < c^-_0 \ll 1$, to obtain that
$$
\partial_y u \leq - \frac{3}{4} g' (0^+) 
\qquad \text{ in } R_{\gamma \overline{r}}.
$$ 
Comparing the inequality
above to \eqref{soon a contradiction}, we obtain the desired contradiction.
\end{proof}

\begin{section}{Frequency Formula}

In this section we prove a frequency formula, 
which will allow us to study the blow up profiles of solutions $u$ of \eqref{defcritptn}.
To this purpose, assuming that $(0, 0)$ is a free boundary point for $u$,
and that $u (x, 0) \geq 0$ in a neighborhood of $0$ (cf. Proposition \ref{phaseseparation}), 
we investigate the regularity properties of 
the function $v: \mathbb{R}^n \times [-A,A] \to \mathbb{R}$ defined by \eqref{v intro}. 

Throughout this section we assume that the hypotheses of Theorem~\ref{opt_reg_theorem}
are satisfied, that $v$ is given by \eqref{v intro} where $u$ is a solution of \eqref{defcritptn},
that $(0, 0)$ is a free boundary point for $v$,
and that $v (x, 0) \geq 0$ for every $x \in B^n_{r_0}$, 
where $r_0$ is given by Proposition~\ref{phaseseparation}. 
Therefore, $v$ satisfies:
\begin{equation*} 
\begin{cases}
\Delta v = 0  \, &\textnormal{ in } \, B_{r_0} \setminus \{ y = 0 \}, \\
v \geq 0 \, 
&\textnormal{ on }\, B_{r_0}^n, \\
\partial_{y} v \leq g' (2 v) - g' (0^+) \, 
&\textnormal{ on }\, B_{r_0}^n, \\
\partial_{y} v = g' (2 v) - g' (0^+) 
\, &\textnormal{ on }\, B_{r_0}^n \cap \{ v > 0 \}.
\end{cases}
\end{equation*}
Since $v$ is even with respect to the hyperplane $\{ y = 0 \}$, we have
\begin{equation} \label{symmetry of v}
v_{\scriptscriptstyle RT} = v_{\scriptscriptstyle LT} \quad \text{ and } 
\quad - \frac{\partial v_{\scriptscriptstyle LT}}{\partial y}
= \frac{\partial v_{\scriptscriptstyle RT}}{\partial y} \leq 0 \qquad \text{ on } B^n_{r_0}. 
\end{equation}
First of all we observe that
\[
\partial_{y} v (x, 0 )= g' (2 v (x, 0 )) - g' (0^+)
\leq 2 \| g '' \| v (x, 0 ) \leq \frac{2 \| g '' \| L_A}{1-2A\| g '' \|_{L^{\infty}}} r_0 = : C_0,
\]
for every $x \in B^n_{r_0}$, 
where we used the definition of $v$, Lemma~\ref{Lip2}, 
and the fact that $(0,0)$ is a free boundary point.
From the above inequality it follows that the function
\begin{equation} \label{def v tilde}
\widetilde{v} (x, y):= v (x, y) - C_0 |y| 
\end{equation}
is superharmonic in $B_{r_0}$.
Indeed, let $\varphi \in C^{\infty}_c (B_{r_0})$ with $\varphi \geq 0$.
Then, using the fact that $\widetilde{v}$ is harmonic in $B_{r_0} \setminus \{ y = 0\}$
\begin{align*}
&\int_{B_{r_0}} \widetilde{v} \Delta \varphi \, dz 
= \int_{B_{r_0} \cap \{ y > 0\}} \widetilde{v}
 \Delta \varphi \, dz + \int_{B_{r_0} \cap \{ y < 0\}} 
 \widetilde{v} \Delta \varphi \, dz\\
&= \int_{B_{r_0} \cap \{ y > 0\}} \text{div} ( 
\widetilde{v} \, \nabla \varphi ) \, dz 
+ \int_{B_{r_0} \cap \{ y > 0\}} \text{div} ( \widetilde{v} \,  \nabla \varphi ) \, dz \\
&- \int_{B_{r_0} \cap \{ y > 0\}} \nabla \widetilde{v} \cdot \nabla \varphi \, dz 
- \int_{B_{r_0} \cap \{ y < 0\}} \nabla \widetilde{v} \cdot \nabla \varphi \, dz\\
&= \int_{B_{r_0}^n} ( \widetilde{v}_{\scriptscriptstyle LT} 
- \widetilde{v}_{\scriptscriptstyle RT} ) \frac{\partial \varphi}{\partial y} \, d \mathcal{H}^n
- \int_{B_{r_0} \cap \{ y > 0\}} \text{div} (  \varphi \nabla \widetilde{v} ) \, dz 
- \int_{B_{r_0} \cap \{ y < 0\}} \text{div} (  \varphi \nabla \widetilde{v} ) \, dz \\
&= \int_{B^n_{r_0}} \varphi \left( \frac{\partial \widetilde{v}_{\scriptscriptstyle RT}}{\partial y} - \frac{\partial \widetilde{v}_{\scriptscriptstyle LT}}{\partial y}   \right) 
= 2  \int_{B^n_{r_0}} \varphi \frac{\partial \widetilde{v}_{\scriptscriptstyle RT}}{\partial y} \, d \mathcal{H}^n \leq 0,
\end{align*}
where we used \eqref{symmetry of v}.
We can now state the main result of the section.
\begin{proposition} \label{9}
Let $F_v : (0, \infty) \to [0, \infty)$ be given by
$$
F_v (r):= \int_{\partial B_r} v^2 \, d \mathcal{H}^n,
$$
let $r_0$ be given by Proposition~\ref{phaseseparation}, and set
\begin{equation*}
\Phi_v (r):= 
r \frac{d}{dr} \log \left( \max\{ F_v (r) , r^{n+4} \} \right).
\end{equation*}
Then, there exist $0 < \overline{r}_0 \leq r_0$, and a positive constant $C$, 
such that the function $r \mapsto \Phi_v (r) e^{C r}$
is monotone nondecreasing for $r \in (0,\overline{r}_0)$.
In particular, there exists
$$
\Phi_v (0^+) = \lim_{r \to 0^+} \Phi_v (r).
$$
\end{proposition}
Before giving the proof, we need several auxiliary lemmas.
When integrating along the boundary of a smooth $(n+1)$-dimensional set, we will 
denote by $\nu$ the outer unit normal, and by $v_{\nu}$ the derivative of $v$ 
along $\nu$. We will denote the tangential gradient of $v$ by $\nabla_{\tau} v$, 
so that $\nabla_{\tau} v = \nabla v - v_{\nu} \nu$.

The next lemma is an adaptation of \cite[Lemma 7.8]{CSS}.
\begin{lemma} \label{5}
For every $r \in (0, r_0)$
\begin{align*}
(n-1) \int_{B_r} |\nabla v|^2 \, dz = r \int_{\partial B_r} 
\left[ |\nabla_{\tau} v |^2 - v_{\nu}^2 \right] \, d \mathcal{H}^n
+ 4 \int_{B^n_r} ( g' (2 v) - g' (0^+) )(x \cdot \nabla_{\tau} v) \, d x.
\end{align*}

\end{lemma}

\begin{proof}
Since $\Delta v=0$ in $B_{r_0} \setminus \{ y = 0 \}$, there we have
\begin{align}
&\text{div} \left[ |\nabla v|^2 z - 2 (z \cdot \nabla v) \nabla v \right] \nonumber \\
&=  (n+1) |\nabla v|^2 + 2 (D^2 v \cdot \nabla v) \cdot z - 2 (z \cdot \nabla v) \Delta v
- 2 \nabla v \cdot \left( \nabla v + (D^2 v \cdot z) \right) \nonumber \\
&=  (n+1) |\nabla v|^2 - 2 |\nabla v|^2  \nonumber
= (n-1) |\nabla v|^2.
\end{align}
Then, 
\begin{align}
& (n-1) \int_{B_r} |\nabla v|^2 \, dz 
= \int_{B_r \cap \{ y > 0\} } \text{div} \left[ |\nabla v|^2 z - 2 (z \cdot \nabla v) \nabla v \right] \, dz  \nonumber \\
&+ \int_{B_r \cap \{ y < 0\} } \text{div} \left[ |\nabla v|^2 z - 2 (z \cdot \nabla v) \nabla v \right] \, dz \nonumber \\
&= \int_{ \partial (B_r \cap \{ y > 0\}) } 
\left[ |\nabla v|^2 (z \cdot \nu) - 2 (z \cdot \nabla v) (\nabla v \cdot \nu) \right] \, d \mathcal{H}^n \nonumber \\
&+ \int_{ \partial (B_r \cap \{ y < 0\}) } 
\left[ |\nabla v|^2 (z \cdot \nu) - 2 (z \cdot \nabla v) (\nabla v \cdot \nu) \right] \, d \mathcal{H}^n. \label{rel1}
\end{align}
Recalling that $z = r \nu$ on $\partial B_r$,
we get 
\begin{align}
&\int_{ \partial (B_r \cap \{ y > 0\}) }  
\left[ |\nabla v|^2 (z \cdot \nu) - 2 (z \cdot \nabla v) (\nabla v \cdot \nu) \right] \, d \mathcal{H}^n \nonumber \\
&= \int_{\partial B_r \cap \{ y > 0\}} 
\left[ |\nabla v|^2 (z \cdot \nu) - 2 (z \cdot \nabla v) (\nabla v \cdot \nu) \right] \, d \mathcal{H}^n \nonumber \\
&\qquad+ \int_{B^n_r} 
\left[ |\nabla v_{\scriptscriptstyle RT}|^2 (x \cdot \nu) - 2 (x \cdot \nabla v_{\scriptscriptstyle RT}) (\nabla v_{\scriptscriptstyle RT} \cdot \nu) \right] \, d x \nonumber \\
&= r \int_{\partial B_r \cap \{ y > 0\}} \left[ |\nabla v|^2 - 2 v_{\nu}^2 \right] \, d \mathcal{H}^n
+ \int_{B^n_r} 
\left[ |\nabla v_{\scriptscriptstyle RT}|^2 (x \cdot \nu) - 2 (x \cdot \nabla v_{\scriptscriptstyle RT}) (\nabla v_{\scriptscriptstyle RT} \cdot \nu) \right] \, d x \nonumber \\
&= r \int_{\partial B_r \cap \{ y > 0\}} \left[ | \nabla_{\tau} v |^2 - v_{\nu}^2 \right] \, d \mathcal{H}^n
+ \int_{B^n_r} 2 (x \cdot \nabla v_{\scriptscriptstyle RT}) \, \partial_{y} v_{\scriptscriptstyle RT} \, d x \nonumber \\
&= r \int_{\partial B_r \cap \{ y > 0\}} \left[ | \nabla_{\tau} v |^2 - v_{\nu}^2 \right] \, d \mathcal{H}^n
+ \int_{B^n_r} 2 (x \cdot \nabla_{x} v_{\scriptscriptstyle RT}) \, \partial_{y} v_{\scriptscriptstyle RT} \, d x. \label{rel2}
\end{align}
Similarly, 
\begin{align}
&\int_{ \partial (B_r \cap \{ y < 0\}) }  
\left[ |\nabla v|^2 (z \cdot \nu) - 2 (z \cdot \nabla v) (\nabla v \cdot \nu) \right] \, d \mathcal{H}^n \nonumber \\
&= r \int_{\partial B_r \cap \{ y < 0\}} \left[ | \nabla_{\tau} v |^2 - v_{\nu}^2 \right] \, d \mathcal{H}^n
- \int_{B^n_r} 2 (x \cdot \nabla_{x} v_{\scriptscriptstyle LT}) \, \partial_{y} v_{\scriptscriptstyle LT} \, d x. \label{rel3}
\end{align}
Combining \eqref{rel1}, \eqref{rel2}, and \eqref{rel3} we obtain
\begin{align*}
& (n-1) \int_{B_r} |\nabla v|^2 \, dz \\
&= r \int_{\partial B_r} \left[ | \nabla_{\tau} v |^2 - v_{\nu}^2 \right] \, d \mathcal{H}^n
+ 2 \int_{B^n_r} \left[ (x \cdot \nabla_{x} v_{\scriptscriptstyle RT}) \, \partial_{y} v_{\scriptscriptstyle RT} - (x \cdot \nabla_{x} v_{\scriptscriptstyle LT}) \, \partial_{y} v_{\scriptscriptstyle LT} \right] \, d x.
\end{align*}
Then, thanks to \eqref{symmetry of v} and the equation satisfied by $v$, we conclude that
\begin{align*}
&(n-1) \int_{B_r} |\nabla v|^2 \, dz 
= r \int_{\partial B_r} \left[ | \nabla_{\tau} v |^2 - v_{\nu}^2 \right] \, d \mathcal{H}^n
+ 4 \int_{B^n_r} \partial_{y} v_{\scriptscriptstyle RT} (x \cdot \nabla_{x} v) \, d x \\
&= r \int_{\partial B_r} \left[ | \nabla_{\tau} v |^2 - v_{\nu}^2 \right] \, d \mathcal{H}^n
+ 4 \int_{B^n_r} ( g' (2 v) - g' (0^+) )(x \cdot \nabla_{x} v) \, d x.
\end{align*}
\end{proof}
We will also need the following lemma.
\begin{lemma} \label{2}
For every $r \in (0, r_0)$
$$
\int_{B_r} | \nabla v |^2 \, d z 
=  \int_{\partial B_r} v \, v_{\nu} \, d \mathcal{H}^n
- 2  \int_{B^n_r} v ( g' (2v) - g' (0^+) ) \, d x.
$$

\end{lemma}

\begin{proof}
Since $v$ is harmonic in $B_r \setminus   \{ y = 0 \}$,
\begin{align*}
&\int_{B_r} \textnormal{div} (v \nabla v) \, d z
= \int_{B_r \cap \{ y > 0\}} \textnormal{div} (v \nabla v) \, d z 
+ \int_{B_r \cap \{ y < 0\}} \textnormal{div} (v \nabla v) \, d z\\
&= \int_{B_r \cap \{ y > 0\}}  | \nabla v |^2 \, d z
+ \int_{B_r \cap \{ y < 0\}} | \nabla v |^2 \, d z \\
&+ \int_{B_r \cap \{ y > 0\}}  v \Delta v \, d z
+ \int_{B_r \cap \{ y < 0\}}  v \Delta v \, d z
= \int_{B_r} | \nabla v |^2 \, d z
\end{align*}
On the other hand, applying the divergence theorem in each half-sphere 
\begin{align*}
&\int_{B_r} \textnormal{div} (v \nabla v) \, d z
= \int_{\partial (B_r \cap \{ y > 0\})} v \, v_{\nu} \, d \mathcal{H}^n 
+ \int_{\partial (B_r \cap \{ y < 0\})} v \, v_{\nu} \, d \mathcal{H}^n\\
&= \int_{\partial B_r} v \, v_{\nu} \, d \mathcal{H}^n
+  \int_{B^n_r} \left[ v_{\scriptscriptstyle LT} \, \partial_y v_{\scriptscriptstyle LT} - v_{\scriptscriptstyle RT} \, \partial_y v_{\scriptscriptstyle RT} \right] \, d x \\
&=\int_{\partial B_r} v \, v_{\nu} \, d \mathcal{H}^n
- 2  \int_{B^n_r} v \, \partial_y v_{\scriptscriptstyle RT}  \, d x \\
&= \int_{\partial B_r} v \, v_{\nu} \, d \mathcal{H}^n
- 2  \int_{B^n_r} v ( g' (2v) - g' (0^+) ) \, d x,
\end{align*}
where we also used \eqref{symmetry of v}.
Comparing the last two chains of equalities we conclude.
\end{proof}
We now start by differentiating $F_v(r)$.

\begin{lemma} \label{1}
For every $r \in (0, r_0)$
\begin{align*}
F_v' (r) &= \int_{\partial B_r} 2 \, v \, v_{\nu} \, d \mathcal{H}^n + \frac{n}{r} F_v (r) \\
&= 2 \int_{B_r} | \nabla v |^2 \, d z + 4  \int_{B^n_r} v ( g' (2v) - g' (0^+) ) \, d x
+ \frac{n}{r} F_v (r).
\end{align*}

\end{lemma}

\begin{proof}
Writing the integral in polar coordinates and differentiating we obtain
\begin{align*}
F_v' (r) &= \frac{d}{dr} \left[ \int_{\mathbb{S}^{n}} v^2 (r \omega) \, r^n d \mathcal{H}^n (\omega) \right] \\
&= \int_{\mathbb{S}^{n}} 2 \, v (r \omega) \nabla v (r \omega) \cdot \omega  \, r^n \mathcal{H}^n (\omega)
+ n \int_{\mathbb{S}^{n}} v^2 (r \omega) \, r^{n-1} d \mathcal{H}^n (\omega) \\
&= \int_{\partial B_r} 2 \, v v_{\nu}  \, d \mathcal{H}^n
+ \frac{n}{r} \int_{\partial B_r} v^2 (z) \, d \mathcal{H}^n \\
&= 2 \int_{B_r} | \nabla v |^2 \, d z + 4  \int_{B^n_r} v ( g' (2v) - g' (0^+) ) \, d x
+ \frac{n}{r} F_v (r),
\end{align*}
where the last equality follows by Lemma \ref{2}.
\end{proof}
We now state a trace inequality, whose proof can be found in \cite{FKS}.
\begin{lemma} \label{11}
For any $r > 0$ and any function $w \in W^{1,2} (B_r)$ we have
$$
\int_{\partial B_r} | w - \overline{w} |^2 \, d \mathcal{H}^n 
\leq C r \int_{B_r} | \nabla w |^2 \, d z,   
$$
where 
$$
\overline{w}:= \fint_{\partial B_r} w \, d \mathcal{H}^{n},
$$
and $C$ is a constant depending only on the dimension $n$.
\end{lemma}
In the following we will need an improvement of Lemma~\ref{11}, 
which can be obtained using the fact that $v$ is superharmonic. 
\begin{lemma} \label{10}
There exists a constant $C$, depending only on $n$, such that for any $r \in (0, r_0)$
$$
F_v (r) = \int_{\partial B_r} v^2 \, d \mathcal{H}^n
\leq C r \int_{B_r} | \nabla v |^2 \, d z, 
$$
and
$$
\int_{B^n_r} v^2 \, d \mathcal{H}^n
\leq C r \int_{B_r} | \nabla v |^2 \, d z.
$$
\end{lemma}

\begin{proof}
Let us start by proving the first inequality.
Since $v \in W^{1,2} (B_r)$, by Lemma \ref{11}
\begin{equation} \label{qwsa}
F_v (r) = \int_{\partial B_r} v^2 \, d \mathcal{H}^n
\leq C r \int_{B_r} | \nabla v |^2 \, d z
+ \overline{v} \int_{\partial B_r} v \, d \mathcal{H}^n.
\end{equation}
Let now $\widetilde{v}$ be given by \eqref{def v tilde}.
 Since $\widetilde{v}$ is superharmonic,   
$$
0 = \widetilde{v} (0) 
\geq \overline{\widetilde{v}} 
= \fint_{\partial B_r} \widetilde{v} \, d \mathcal{H}^n
= \fint_{\partial B_r} v \, d \mathcal{H}^n 
= \fint_{\partial B_r} v^+ \, d \mathcal{H}^n + \fint_{\partial B_r} v^- \, d \mathcal{H}^n.
$$
The above inequality implies that
\begin{equation} \label{lessthan}
\int_{\partial B_r} v^+ \, d \mathcal{H}^n 
\leq - \int_{\partial B_r} v^- \, d \mathcal{H}^n
= \int_{\partial B_r} |v^-| \, d \mathcal{H}^n.
\end{equation}
Since $v (x, 0) \geq 0$, we have $v^- (x, 0) = 0$. Thus, by rescaling,
$$
\frac{r\int_{B_r}|\nabla v^-|^2dz}{\int_{\partial B_r} (v^-)^2d \mathcal{H}^n}\geq \min\biggl\{
\frac{\int_{B_1}|\nabla w|^2dz}{\int_{\partial B_1} w^2d \mathcal{H}^n}\,:\,w:B_1\to \R,\,\,w|_{B_1^n}=0\biggr\}=:c_n>0,
$$
where the positivity of $c_n$ follows by a standard compactness argument (actually, by spectral analysis theory, the minimum is attained by $w(x,y)=y$, thus $c_n=1$).

Hence, by H\"older inequality,
$$
\fint_{\partial B_r} | v^- | \, d \mathcal{H}^n \leq \biggl(\fint_{\partial B_r} (v^- )^2 \, d \mathcal{H}^n\biggr)^{1/2}\leq 
\biggl(c_n^{-1}r^2\fint_{B_r} |\nabla v^-|^2 \, dz\biggr)^{1/2},
$$
that combined with \eqref{lessthan} yields
\begin{align*}
\int_{\partial B_r} | v | \, d \mathcal{H}^n
\leq 2 \int_{\partial B_r} | v^- | \, d \mathcal{H}^n 
\leq C r^{\frac{n+1}{2}} \left( \int_{B_r} | \nabla v |^2 \, dz \right)^{1/2}. 
\end{align*}
Finally, plugging this into \eqref{qwsa} we get
\begin{align*}
F_v (r) &\leq C \biggr[r \int_{B_r} | \nabla v |^2 \, d z
+ \frac{1}{r^n} \left( \int_{\partial B_r} | v | \, d \mathcal{H}^n \right)^2 \biggr]
\leq C r \int_{B_r} | \nabla v |^2 \, d z, 
\end{align*}
which proves the first inequality of the statement.

To show the second inequality, it is enough to observe that
$$
\min\biggl\{
\frac{\int_{B_1}|\nabla w|^2dz+\int_{\partial B_1} w^2d \mathcal{H}^n}{\int_{B_1^n} w^2d \mathcal{H}^n}\,:\,w:B_1\to \R\biggr\}=:c_n'>0,
$$
where again positivity of $c_n'$ follows by a standard compactness argument.
By rescaling, this implies that 
$$
\int_{B^n_r} v^2 \, d \mathcal{H}^n
\leq \frac{1}{c_n'} \biggl(r \int_{B_r} | \nabla v |^2 \, d z+\int_{\partial B_r} v^2 \, d \mathcal{H}^n\biggr),
$$
and the result follows by the first inequality of the statement.
\end{proof}
Before proving Proposition~\ref{9} we need another lemma.
\begin{lemma} \label{8}
There exists $\overline{r}_0  \in (0,  r_0 )$ 
and a positive constant $C= C (n)$ such that, whenever $F_v (r) > r^{n+4}$, we have
$$
\int_{B_r} | \nabla v |^2 \, d z \leq 2 \int_{\partial B_r} v \,  v_{\nu} \, d \mathcal{H}^n,
$$
and 
$$
\int_{\partial B_r} v \,  v_{\nu} \, d \mathcal{H}^n > C r^{n+3},
$$
for every $ 0 < r < \overline{r}_0$.
\end{lemma}

\begin{proof}
Suppose that $F_v (r) > r^{n+4}$. Then, by Lemma \ref{10},
$$
r^{n+4} < F_v (r) \leq C r \int_{B_r} | \nabla v |^2 \, d z, 
$$
which implies
\begin{equation} \label{1 over C}
\int_{B_r} | \nabla v |^2 \, d z > \frac{r^{n+3}}{C}. 
\end{equation}
Thanks to Lemma \ref{2} and Lemma \ref{10}, for $r$ sufficiently small we have
\begin{align}\label{vnu Dir}
&\int_{\partial B_r} v \,  v_{\nu} \, d \mathcal{H}^n
= \int_{B_r} | \nabla v |^2 \, d z + 2  \int_{B_r^n} v \, ( g' (2v) - g' (0^+) ) \, d x \nonumber\\
&\geq \int_{B_r} | \nabla v |^2 \, d z - 4 \| g '' \|_{L^{\infty}} \int_{B_r^n} v^2 \, d x  \\
&\geq  (1 - 4 r C  \| g '' \|_{L^{\infty}} ) \int_{B_r} | \nabla v |^2 \, d z \geq \frac{1}{2} \int_{B_r} | \nabla v |^2 \, d z.\nonumber
\end{align}
This shows the first inequality which, together with \eqref{1 over C}, allows us to conclude.
\end{proof}
We are now ready to prove Proposition~\ref{9}.
\begin{proof}[Proof of Proposition \ref{9}]
Since $r\mapsto \max\{ F_v (r) , r^{n+4} \}$ is a semiconvex function (being the maximum between two smooth functions)
and $\Phi_v (r) = n+4$ on the region where $\max\{ F_v (r) , r^{n+4} \} = r^{n+4}$, it suffices to prove the monotonicity of $\Phi_v(r)e^{Cr}$ in the open set $\{r\,:\,F_v (r) > r^{n+4}\}$.

Note that, thanks to Lemma \ref{1},
\begin{equation} \label{used also later}
\Phi_v  (r) = r  \frac{F'_v (r)}{F_v (r)}  
= \frac{2 r}{F_v (r)} \int_{\partial B_r} v \, v_{\nu} \, d \mathcal{H}^n + n.
\end{equation}
Setting $\Psi_v (r):= \Phi_v (r) - n$, the logarithmic derivative of $\Psi_v$ is given by
\begin{align}
\frac{\Psi_v' (r)}{\Psi_v (r)} &= \frac{d}{dr} \log \left( \frac{r}{F_v (r)} \int_{\partial B_r} v \, v_{\nu} \, d \mathcal{H}^n \right) \nonumber \\
&= \frac{d}{dr} \left( \log r  + \log \int_{\partial B_r} v \, v_{\nu} \, d \mathcal{H}^n
- \log ( F_v (r) ) \right) \nonumber\\
&= \frac{1}{r}  
+ \frac{ \displaystyle \frac{d}{d r} \int_{\partial B_r} v \, v_{\nu} \, d \mathcal{H}^n}{\displaystyle \int_{\partial B_r} v \, v_{\nu} \, d \mathcal{H}^n}
- \frac{2 \displaystyle \int_{\partial B_r} v \, v_{\nu} \, d \mathcal{H}^n}{F_v (r)}
- \frac{n}{r}, \label{6}
\end{align}
where we used again Lemma~\ref{1}.
We now divide the remaining part of the proof into several steps.
In the following, it will be convenient to define
\begin{align*}
I_1(r): &= 2 \int_{\partial B^n_r} v ( g' (2v)  - g' (0^+) ) \, d \mathcal{H}^{n-1} , \\
I_2(r): &= - \frac{2(n-1)}{r}  \int_{B_r^n} v ( g' (2v)  - g' (0^+) ) \, d x , \\
I_3(r): &= - \frac{4}{r}  \int_{B^n_r} ( g' (2v)  - g' (0^+) )
(x \cdot \nabla_{x} v) \, d x.
\end{align*}

\vspace{.2cm}

\noindent
\textbf{Step 1.} We show that
\begin{align*}
\frac{d}{d r} \int_{\partial B_r} v \, v_{\nu} \, d \mathcal{H}^n 
= \frac{n-1}{r} 
\int_{\partial B_r} v \, v_{\nu} \, d \mathcal{H}^n 
+ 2 \int_{\partial B_r} v_{\nu}^2 \, d \mathcal{H}^n
+ I_1(r) + I_2(r) + I_3(r). 
\end{align*}
Indeed, thanks to Lemma \ref{2}, 
\begin{align}
&\frac{d}{d r} \int_{\partial B_r} v \, v_{\nu} \, d \mathcal{H}^n
= \frac{d}{d r} \int_{B_r} | \nabla v |^2 \, d z 
+ 2 \frac{d}{d r} \int_{B_r^n} v ( g' (2v)  - g' (0^+) ) \, d x \nonumber \\
&= \int_{\partial B_r} | \nabla v |^2 \, d \mathcal{H}^n + I_1(r). \label{4}
\end{align}
Using first Lemma \ref{5} and then Lemma \ref{2}, 
the first integral in the last expression can be written as  
\begin{align*}
&\int_{\partial B_r} |\nabla v|^2 \, d \mathcal{H}^n
= \frac{n-1}{r} \int_{B_r} |\nabla v|^2 \, dz
+ 2 \int_{\partial B_r} v_{\nu}^2 \, d \mathcal{H}^n
+ I_3 (r) \\
&=  \frac{n-1}{r} 
\int_{\partial B_r} v \, v_{\nu} \, d \mathcal{H}^n  + I_2 (r)
+  2 \int_{\partial B_r} v_{\nu}^2 \, d \mathcal{H}^n
+ I_3 (r).
\end{align*}
Inserting this last equality into \eqref{4}, we obtain the claim.

\vspace{.2cm}

\noindent
\textbf{Step 2.} We prove that
\begin{align*}
I_1(r) &+ I_2(r) + I_3(r) 
\geq -  C  \int_{B_r} | \nabla v |^2 \, d z - C r^{ n + \frac{7}{2}}.
\end{align*}
Indeed,
\begin{align*}
\frac{1}{2} I_3 ( r ) &= - \frac{2}{r}  \int_{B^n_r } ( g' (2v)  - g' (0^+) ) (x \cdot \nabla_{x} v) \, d x \\
&= - \frac{2}{r}  \int_{B^n_r  } 
\text{div}_{x} \big( v ( g' (2v)  - g' (0^+) ) x \big) \, d x 
+  \frac{2}{r}  \int_{B^n_r  } v ( g' (2v)  - g' (0^+) ) (\text{div}_{x} x ) \, d x \\
& \hspace{.4cm}+ \frac{2}{r}  \int_{B^n_r  } v \, x \cdot  \nabla_{x} ( g' (2v) )  \, d x \\
&= - I_1 ( r ) - I_2 ( r ) +  \frac{2}{r}  \int_{B^n_r  } v ( g' (2v)  - g' (0^+) ) \, d x
+ \frac{2}{r}  \int_{B^n_r  } 2 v (x \cdot \nabla_{x} v) g'' (2v)  \, d x.
\end{align*}
Therefore, 
\begin{align}
I_1(r) &+ I_2(r) + I_3(r) \label{list1} \\
&= \frac{1}{2} I_3 ( r ) +  \frac{2}{r}  \int_{B^n_r  } v ( g' (2v)  - g' (0^+) ) \, d x
+ \frac{2}{r}  \int_{B^n_r  } 2 v (x \cdot \nabla_{x} v) g'' (2v)  \, d x. \nonumber
\end{align}
Let us first estimate the second term in the right hand side of the  identity above. 
Thanks to Lemma~\ref{10},
\begin{align}
&\frac{2}{r}  \int_{B^n_r } v ( g' (2v)  - g' (0^+) ) \, d x
 \geq - \frac{4 \| g'' \|_{L^{\infty}}}{r}  \int_{B^n_r } v^2 \, d x   
\geq -  C  \int_{B_r} | \nabla v |^2 \, d z. \label{list2} 
\end{align}
Let us now estimate the remaining two terms. 
There exists $\tau = \tau (x) \in (0,1)$ such that
\begin{align}
&\frac{1}{2} I_3 ( r ) 
+ \frac{2}{r}  \int_{B^n_r  } 2 v (x \cdot \nabla_{x} v) g'' (2v)  \, d x \nonumber \\
&= \frac{2}{r}  \int_{B^n_r  } 
(x \cdot \nabla_{x} v) \Big( 2 v g'' (2v) - g' (2v)  + g' (0^+) \Big) \, d x \nonumber \\
&= \frac{2}{r}  \int_{B^n_r  }  2 v
(x \cdot \nabla_{x} v) ( g'' (2v) - g'' (2 \tau v) ) \, d x \nonumber \\
&\geq -\frac{8 \| g''' \|_{L^{\infty}}}{r}  \int_{B^n_r  }  v^{2} |x \cdot \nabla_{x} v | \, d x \nonumber \\
&\geq - 8 C  \| g''' \|_{L^{\infty}} \, r^{3 + \frac{1}{2}} \int_{B^n_r }\, d x
\geq - C r^{ n + \frac{7}{2}}, \label{list3}
\end{align}
where we used that, by the optimal regularity of $v$ (see Theorem \ref{opt_reg_theorem}), $|v|\leq Cr^{3/2}$ and $|\nabla v|\leq Cr^{1/2}$.
Combining \eqref{list1}--\eqref{list3}, for $r$ sufficiently small the claim follows.
\vspace{.2cm}

\noindent
\textbf{Step 3.} We conclude.
Recalling that $\Psi_v (r)= \Phi_v (r) - n$, from Step 1 we have
\begin{align}
\frac{\Psi_v'(r)}{\Psi_v(r)} 
&= \frac{1}{r}  - \frac{n}{r} 
+ \frac{ \displaystyle \frac{d}{d r} \int_{\partial B_r} v \, v_{\nu} \, d \mathcal{H}^n}{\displaystyle \int_{\partial B_r} v \, v_{\nu} \, d \mathcal{H}^n}
- \frac{2 \displaystyle \int_{\partial B_r} v \, v_{\nu} \, d \mathcal{H}^n}{F_u (r)} \nonumber \\
&=  \frac{\displaystyle 2 \int_{\partial B_r} v_{\nu}^2 \, d \mathcal{H}^n}
{\displaystyle \int_{\partial B_r} v \, v_{\nu} \, d \mathcal{H}^n}
- \frac{\displaystyle 2 \int_{\partial B_r} v \, v_{\nu} \, d \mathcal{H}^n}{\displaystyle \int_{\partial B_r} v^2 \, d \mathcal{H}^n}
+ \frac{I_1(r) + I_2(r) + I_3(r)}{\displaystyle \int_{\partial B_r} v \, v_{\nu} \, d \mathcal{H}^n}\nonumber \\
&\geq  \frac{I_1(r) + I_2(r) + I_3(r)}{\displaystyle \int_{\partial B_r} v \,  v_{\nu} \, d \mathcal{H}^n}, \nonumber
\end{align}
where in the last step we used H\"older inequality and the fact that, 
by Lemma~\ref{8}, the integral at the denominator is positive.
Then, thanks to Step 2 and Lemma \ref{8} again, we obtain
\begin{align*}
\frac{\Psi_v'(r)}{\Psi_v(r)} 
& \geq -  C  \frac{\displaystyle \int_{B_r} | \nabla v |^2 \, d z}{\displaystyle \int_{\partial B_r} v \, v_{\nu} \, d \mathcal{H}^n}
- C \frac{ r^{ n + \frac{7}{2} } }{\displaystyle \int_{\partial B_r} v \, v_{\nu} \, d \mathcal{H}^n} \\
& \geq  - C - C r^{n + \frac{7}{2} - (n+3)} \geq - C.
\end{align*}
The previous chain of inequalities gives
\begin{align*}
0 \leq \left[ \log \Psi_v (r) + C r \right]'
= \left[ \log \Big( \Psi_v (r) e^{C r}   \Big) \right]'. 
\end{align*}
Recalling that $\Psi_v (r) = \Phi_v (r) - n$, this shows that 
$r \mapsto (\Phi_v (r) - n) e^{C r}$
is increasing, and thus the conclusion.
\end{proof}

\end{section}

\begin{section}{Blow up profiles and regularity of the free boundary}

We are now going to study the blow up profiles of $v$ and the regularity of the free boundary.
As in the previous section, with no loss of generality we will assume that $(0,0)$
is a free boundary point and that $v (x, 0) \geq 0$ for every $x \in B^n_{\overline{r}_0}$, 
where $\overline{r}_0$ is given by Proposition \ref{9}. 

We define, for every $r \in (0,\overline{r}_0)$, the function $v_r : B_1 \to \mathbb{R}$ as
\begin{equation} \label{defin vr}
v_r (z) := \frac{v (r z)}{d_r} , \qquad d_r:= \left( \frac{F_v (r)}{r^n} \right)^{\frac{1}{2}},
\end{equation}
where $F_v$ is as in Proposition~\ref{9}.
Note that  
\begin{equation} \label{bounded trace}
\int_{\partial B_1} v_r^2 \, d \mathcal{H}^n = 1 \qquad \text{ for every } r < \overline{r}_0.
\end{equation}
Next proposition shows which are the possible values of $\Phi_v (0^+)$. 
\begin{proposition} \label{Phi}
Let the assumptions of Theorem~\ref{opt_reg_theorem} be satisfied, 
suppose that $(0, 0)$ is a free boundary point for $u$,
and that $u(x,0)\geq 0$ in $B_{\overline r_0}^n$.
Let $v$ be given by \eqref{v intro}, and let $\Phi_v$ be as in Proposition~\ref{9}.
Then, either $\Phi_v (0^+) = n+3$, or $\Phi_v (0^+) \geq n+4$.
\end{proposition}
We first prove the proposition above in the case 
$$
\liminf_{r \to 0^+} \frac{d_r}{r^2} < + \infty.
$$
\begin{proof}[Proof of Proposition~\ref{Phi} when $\liminf_{r\to 0^+} d_r/r^2 < + \infty$]

\mbox{ }

\smallskip

\noindent
We will show that in this case we always have $\Phi_v (0^+) \geq n+4$.
Indeed, by assumption there exists $C > 0$ such that
$$
\frac{d_r^2}{r^4} = \frac{F_v (r)}{r^{n+4}} \leq C \qquad \text{ for every } r \in (0, 1). 
$$
We then have two possibilities (see also the second part of the proof of \cite[Lemma~6.1]{CSS}). 

\vspace{.2cm}

\noindent
\textbf{Case 1:} there exists a sequence $(r_j)_{j \in \mathbb{N}}$ with $r_j \to 0$ 
such that
$$
F_v (r_j) < r_j^{n+4} \quad \text{ for $j$ sufficiently large}.
$$ 
Then, $\Phi_v (r_j)=  n+4$ for $j$ sufficiently large, and therefore
$\Phi_v (0^+) =  n+4$.

\vspace{.2cm}

\noindent
\textbf{Case 2:} for $r$ sufficiently small
$$
r^{n+4} \leq F_v (r) \leq C r^{n+4}.
$$
Then, 
\begin{equation} \label{logarithm}
( n+4 ) \log r \leq \log F_v (r) \leq \log C + ( n+4 ) \log r.
\end{equation}
Suppose now, by contradiction, that there exists $\eta > 0$ such that
$$
\Phi_v (r) \leq n+4 - \eta \quad \text{ for $r$ sufficiently small},
$$
and let $(r_j)_{j \in \mathbb{N}}$ be a strictly decreasing sequence  with $r_j \to 0$. 
Then, thanks to \eqref{logarithm}, for every $k, l \in \mathbb{N}$ with $k < l$ we have
\begin{align*}
( n+4 ) \log r_k \leq \log F_v (r_k) \quad \text{ and } \quad 
\log F_v (r_l) \leq \log C + ( n+4 ) \log r_l.
\end{align*}
Therefore, by the definition of $\Phi_v$,
\begin{align*}
&( n+4 ) ( \log r_k - \log r_l) - \log C
\leq \log F_v (r_k) - \log F_v (r_l) 
 = \int_{r_l}^{r_k} \frac{d}{dr} \log F_v (r) \, dr \\
&= \int_{r_l}^{r_k} \frac{\Phi_v (r)}{r} \, dr 
\leq (n+4 - \eta) ( \log r_k - \log r_l),  
\end{align*}
which is impossible if we choose $\log r_k - \log r_l \to \infty$.
\end{proof}
In the next proposition we consider the case 
$$
\liminf_{r \to 0^+} \frac{d_r}{r^2} = + \infty.
$$
\begin{proposition} \label{blow up limit}
Let the assumptions of Theorem~\ref{opt_reg_theorem} be satisfied, 
suppose that $(0, 0)$ is a free boundary point for $u$,
and that $u(x,0)\geq 0$ in $B_{\overline r_0}^n$.
Let $v$ be given by \eqref{v intro}, and let $F_v$ and $\Phi_v$ be as in Proposition~\ref{9}.
Define $v_r$ as in  \eqref{defin vr}, and assume that
$$
\liminf_{r \to 0^+} \frac{d_r}{r^2} = + \infty.
$$
Then, there exists a sequence $(r_k)_{k \in \mathbb{N}}$ with $r_k \to 0$,
and a homogeneous function $v_{\infty} \in W^{1,2} (B_1)$
with homogeneity degree $1/2(\Phi_v (0^+) - n)$,
such that
$$
v_{r_k} \rightharpoonup v_{\infty} \quad \text{ weakly in } W^{1,2} (B_1),
$$
and
\begin{equation} \label{convunif}
v_{r_k} \to v_{\infty} \text{ in $C^{1, \gamma}$ on compact subsets of } B_{1} \cap \{ y \geq 0 \},
\end{equation}
for any $\gamma \in (0, 1/2)$.
Moreover, $v_{\infty}$ satisfies the classical Signorini problem in $B_1$
and is even with respect to $y$:
\begin{equation} \label{signorini}
\begin{cases}
\Delta v_{\infty} = 0 & \text{ in } B_1 \setminus \{ y = 0 \}, \\
v_{\infty} \geq 0 & \text{ on } B^n_1, \\
\partial_y v_{\infty} \leq 0 & \text{ on } B^n_1, \\
v_{\infty} \partial_y v_{\infty} = 0 & \text{ on } B^n_1, \\
v_{\infty} (x, - y) = v_{\infty} (x, y) & \text{ in } B_1. \\
\end{cases}
\end{equation}
Finally, it holds that $\Phi_v (0^+) = n + 3$ and that, up to a multiplicative constant 
and to a change of variables, we have 
\begin{equation} \label{case a}
v_{\infty} (x, y) = \rho^{3/2} \cos \frac{3}{2} \theta, 
\end{equation}
where $\rho^2 = x_n^2 + y^2$ and $\tan \theta = y/x_{n}$.
\end{proposition}
\begin{proof}[Proof of Proposition~\ref{blow up limit} and conclusion of 
proof of Proposition~\ref{Phi}]

\mbox{ }

\smallskip

Since $d_r/r^2 \to \infty$, for $r$ sufficiently small we have $F_v (r) > r^{n+4}$.
Then, thanks to Proposition~\ref{9} and by \eqref{used also later}, 
for $r < \overline{r}_0$ we have
\begin{align*}
\Phi_v (\overline{r}_0) e^{C (\overline{r}_0 - r)}  \geq \Phi_v (r)
= 2 r \frac{\displaystyle \int_{\partial B_r} v \, v_{\nu} \, d \mathcal{H}^n}
{\displaystyle \int_{\partial B_r} v^2 \, d \mathcal{H}^n} + n.
\end{align*}
Therefore, by \eqref{vnu Dir} and the definition of $v_r$,
\begin{align*}
\Phi_v (\overline{r}_0) e^{C (\overline{r}_0- r)} - n
\geq 2 r \,\frac{ \displaystyle \int_{\partial B_r} v \, v_{\nu} \, d \mathcal{H}^n}
{\displaystyle \int_{\partial B_r} v^2 \, d \mathcal{H}^n} 
= 2 r \,\frac{\frac12 \displaystyle \int_{B_r} | \nabla v |^2 \, d z }{\displaystyle \int_{\partial B_r} v^2 \, d \mathcal{H}^n} =\int_{B_1} | \nabla v_r |^2 \, d z,
\end{align*}
for $r$ sufficiently small.
Consider now a sequence $r_k \to 0$. 
By the previous inequality and thanks to \eqref{bounded trace}, the sequence 
$(v_{r_k})_{k \in \mathbb{N}}$ is bounded in $W^{1,2} (B_1)$. Thus, up to subsequences,
$$
v_{r_k} \rightharpoonup v_{\infty} \text{ weakly in } W^{1,2} (B_1),
$$
for some $v_{\infty} \in W^{1,2} (B_1)$.
Thanks to the uniform $C^{1,1/2}$ regularity for solutions, we also have that \eqref{convunif}
holds.
Let us show that $v_{\infty} \partial_y v_{\infty} = 0$ on $B^n_1$, 
since the other conditions in \eqref{signorini} are a direct consequence of \eqref{convunif}.
Recalling the definition of $v_{r_k}$, from the identity 
$$
v (rx, 0) \left[ \partial_y v (rx, 0) - g' (2 v (rx, 0)) + g' (0^+) \right] = 0 \quad \text{ for every } x \in B_1^n
$$
it follows that, for every $k \in \mathbb{N}$,
\begin{equation} \label{equation v_k}
v_{r_k} (x, 0) \left[ \partial_y v_{r_k} (x, 0) - \frac{r_k}{d_{r_k}} ( g' ( 2 d_{r_k} v_{r_k} (x, 0)) - g' (0^+) ) \right]=0.
\end{equation}
Thanks to \eqref{convunif}, since 
$$
\frac{r_k}{d_{r_k}} | g' ( 2 d_{r_k} v_{r_k} (x, 0) ) - g' (0^+) | \leq 2 r_k v_{r_k} (x, 0) \| g'' \|_{L^{\infty}} \stackrel{k\to \infty}{\to} 0
\qquad \text{ for every } x \in B_1^n, 
$$ 
taking the limit in \eqref{equation v_k} we obtain 
$$
v_{\infty} \partial_y v_{\infty} = 0  \qquad \text{ in } B^n_1.
$$
To show that $v_\infty$ is homogeneous, let us first prove that $\Phi_{v_{\infty}}$ is constant for $r$ sufficiently small.
Indeed, let $r < s \ll 1$. 
A direct calculation shows that
$$
\Phi_{v_{r_k}} (r) - \Phi_{v_{r_k}} (s) = \Phi_v (r_k r) - \Phi_v (r_k s) 
\qquad \text{ for every } k \in \mathbb{N}. 
$$ 
Thanks to \eqref{convunif}, taking the limit as $k \to \infty$ we obtain
$$
\Phi_{v_{\infty}} (r) - \Phi_{v_{\infty}} (s) = 0,
$$
where we used the existence of the limit $\lim_{r\to 0^+}\Phi_v (r)$, which follows from Proposition~\ref{9}.
Since $v_{\infty}$ satisfies the Signorini problem \eqref{signorini}, from \cite[Lemma 1]{ACS} 
it follows that $v_{\infty}$ is homogeneous and that 
$$
\Phi_v (0^+) = 2 \mu + n, 
$$
where $\mu$ is  the homogeneity degree of $v_{\infty}$.
Therefore, 
$$
\mu = \frac{\Phi_v (0^+) - n}{2}.
$$
Arguing as in \cite[Lemma~6.6]{CSS}, one gets 
that $F_v (r) \leq C r^{2 \mu + n}$. Since $d_r/r^2\to \infty$, this implies $\mu<2$,
and one concludes as in \cite[Section 4]{ACS}
that $\mu =3/2$ and that the function $v_{\infty}$ is given by \eqref{case a}.
\end{proof}
\end{section}

\begin{section}{$C^{1, \alpha}$ regularity of the free boundary for $\mu = 3/2$.}

We now study the regularity of the free boundary in the special case in which 
$\Phi_v (0^+) = n + 3$. Note that, by the argument in the previous section,
this corresponds to the case
$$
\liminf_{r \to 0^+} \frac{d_r}{r^2} = + \infty.
$$
 We start by proving the $C^1$ regularity.
\begin{lemma} \label{C1 free}
Let the assumptions of Theorem~\ref{opt_reg_theorem} be satisfied, 
suppose that $(0, 0)$ is a free boundary point for $u$,
and that $u(x,0)\geq 0$ in $B_{\overline r_0}^n$.
Let $v$ be given by \eqref{v intro}, and let $\Phi_v$ be as in Proposition~\ref{9}.
Assume that $\Phi_v (0^+) = n+3$, 
and choose a coordinate system in $\R^n$ such that \eqref{case a} holds true.
Then, for every $\overline{c} > 0$ there exists $\overline{\rho} = \overline{\rho} (\overline{c}) > 0$
with the following property:
For every $\tau \in \mathbb{S}^n \cap \{ y =  0 \}$ with $\tau \cdot \mathbf e_n \geq \overline{c}$ we have
\begin{equation} \label{tau_non_neg}
\partial_{\tau} v (z) \geq 0, \qquad \text{ for every } z \in B_{\overline{\rho}}. 
\end{equation}
In addition, near the origin the free boundary of $v$  is the graph of a $C^1$   function 
$$
x_n = f (x_1, \ldots, x_{n-1}).
$$
 \end{lemma}
Before giving the proof of Lemma \ref{C1 free} we make some useful observation 
on the tangential derivatives of the functions $v_{r_k}$ introduced in the previous section.

Let us fix $\overline{c} > 0$ and $\mathbf e \in \mathbb{S}^n \cap \{ y =  0 \}$ with $\mathbf e \cdot \mathbf e_n = 0$.
Choose now $a \geq \overline{c}$ and $b \in \mathbb{R}$ such that $a^2 + b^2 = 1$, 
and define $h_k : B_1 :\to \mathbb{R}$ as the sequence of functions given by
\begin{equation} \label{h_k}
h_k := \partial_{\tau} v_{r_k} \qquad \text{ for every } k \in \mathbb{N}, 
\end{equation}
where $\tau:= a \mathbf e_n + b \mathbf e$.
For any $\eta \in (0, 1/(8n))$, thanks to \eqref{case a} and \eqref{convunif}, 
there exist $k_0 = k_0 (a, b, \eta)$ and $c_0 = c_0 (a, b , \eta)$ such that
the following properties are satisfied for $k > k_0$:
\begin{itemize}

\item[(i)] $\Delta h_k = 0$ in $B_{2/3} \cap \{ |y| > 0\}$;

\vspace{.2cm}

\item[(ii)] $h_k \geq 0$ in $B_{2/3} \cap \{ |y| > \eta \}$;

\vspace{.2cm}

\item[(iii)] $h_k \geq c_0$ in $B_{2/3} \cap \left\{ |y| > \frac{1}{8 n} \right\}$; 

\vspace{.2cm}

\item[(iv)] $h_k \geq - C \eta^{1/2}$ in $B_{2/3}$,

\end{itemize}
where property (iv) follows from the optimal regularity and (ii).
Let us show that we also have 
\begin{itemize}


\item[(v)] $\partial_y h_k \leq C \eta^{1/2}$ on $B^n_{2/3} \cap \{ h_k \neq 0 \}$.

\end{itemize}
To this aim, first of all observe that $B^n_{2/3} \cap \{ h_k \neq 0 \} \subset B^n_{2/3} \cap \{ v_{r_k} \neq 0 \}$. 
Indeed, if $x \in B^n_{2/3}$ is such that $v_{r_k} (x, 0) = 0$, then
$$
h_k (x, 0) = (\partial_{\tau} v_{r_k} ) (x, 0) = \frac{r_k}{d_{r_k}} ( \partial_{\tau} v ) (r_k x , 0) = 0,
$$
by nonnegativity of $v$ and optimal regularity.

Let now $x \in B^n_{2/3}$ be such that $h_k (x, 0) \neq 0$. Then we have $v_{r_k} (x, 0) > 0$
and, for $k$ sufficiently large,
\begin{align*}
(\partial_y h_k ) (x, 0)
&= \frac{r_k}{d_{r_k}} \partial_{\tau} \left\{ g' (2 v (r_{k} x,0)) - g' (0^+) \right\}
= 2 r_k  g'' (2 v (r_{k} x,0)) h_k  \leq C \eta^{1/2},
\end{align*}
where we used (iv).
We now consider a version of \cite[Lemma 5]{ACS} which is useful for our purposes.
\begin{lemma} \label{lemma h}
Let $0 < \eta < 1/(8n)$, let $C, c_0 > 0$, and let $\sigma : [0, 1] \to [0, \infty)$ 
be a continuous function with $\sigma (0)=0$.
Suppose that $h : B_1 :\to \mathbb{R}$ satisfies the following assumptions:

\begin{itemize}

\item[(i)] $\Delta h = 0$ in $B_1 \cap \{ |y| > 0\}$;

\vspace{.2cm}

\item[(ii)] $h \geq 0$ in $B_1 \cap \{ |y| > \eta \}$;

\vspace{.2cm}

\item[(iii)] $h \geq c_0$ in $B_1 \cap \left\{ |y| > \frac{1}{8 n} \right\}$; 

\vspace{.2cm}

\item[(iv)] $h \geq - \sigma (\eta )$ in $B_1$,

\vspace{.2cm}

\item[(v)] $\partial_y h \leq \sigma (\eta )$ on $B^n_1 \cap \{ h \neq 0 \}$.

\end{itemize}
Then, there exists $\eta_0 = \eta_0 (n, c_0, \sigma)$ such that if $\eta < \eta_0$
we have $h \geq 0$ in $B_{1/2}$.
\end{lemma}

\begin{proof}
Suppose, by contradiction, that there exists $\overline{z} = (\overline{x}, \overline{y}) \in B_{1/2}$
such that $h (\overline{z}) < 0$ (note that, by (iii), this implies $\overline{y} < 1/(8n)$).
We define
$$
Q:= \left\{ (x, y) \in \mathbb{R}^{n+1} : |x - \overline{x}| < \frac{1}{3}, \, \, 0 < | y | < \frac{1}{4 n} \right\},
$$
and
$$
P (x, y) : = |x - \overline{x}|^2 - n y^2, 
$$
and we set
$$
w (z) := h (z) + \delta P (z) - \sigma (\eta) y, 
$$
where $\delta > 0$ will be chosen later.
Note that $w$ is harmonic in $Q$ and
$$
w (\overline{z}) = h (\overline{z}) - \delta \overline{y}^2 - \sigma (\eta ) \overline{y}  < 0.
$$
Therefore, there exists a minimum point $\hat{z} = (\hat{x}, \hat{y}) \in \partial Q$ such that
$$
\min_{z \in \overline{Q}} w (z) = w (\hat{z}) < 0.
$$
We have the following possibilities.

\vspace{.2cm}

\textbf{Case 1.} $\hat{z} \in \partial Q \cap \{ y > 1/(8n)\}$. 
Thanks to (iii), for $\eta$ and $\delta$ sufficiently small we have
$$
w (\hat{z}) \geq c_0 - \frac{\delta}{16 n} - \frac{\sigma (\eta)}{4 n}  > 0,
$$
which is impossible.

\vspace{.2cm}

\textbf{Case 2.} $\hat{z} \in \partial Q \cap \{ \eta \leq y < 1/(8n)\}$. 
Using property (ii) we obtain that for $\eta$ sufficiently small
$$
w (\hat{z}) \geq \delta \left( \frac{1}{9} - \frac{1}{64 n} \right) - \frac{\sigma (\eta)}{8 n}  > 0,
$$
which is impossible.

\vspace{.2cm}

\textbf{Case 3.} $\hat{z} \in \partial Q \cap \{ (x, y) \in \mathbb{R}^{n+1}:  |x - \overline{x} | = \frac{1}{3}, 0 <  y < \eta \}$. 
Thanks to property (iv), for $\eta$ sufficiently small
\begin{align*}
w (\hat{z}) \geq - \sigma (\eta) + \delta \left( \frac{1}{9} - n \eta^2 \right) - \eta \, \sigma (\eta)
= \delta \left( \frac{1}{9} - n \eta^2 \right) - (1 + \eta)  \, \sigma (\eta) > 0,
\end{align*}
which is impossible.

\vspace{.2cm}

\textbf{Case 4.} $\hat{z} \in \partial Q \cap \{ y = 0 \}$. In this case, if $\hat{z} \in \{ h = 0 \}$ we obtain
$$
w (\hat{z}) = \delta P (\hat{z})  = \delta | \hat{x} - \overline{x} |^2  \geq 0,  
$$
which is impossible.
On the other hand, if $\hat{z} \in \{ h \neq 0 \}$, using Hopf Lemma and property (v)
$$
0 < \partial_y w (\hat{z}) = \partial_y h (\hat{z}) - \sigma (\eta) \leq 0,  
$$
which is impossible.
\end{proof}
We are now ready to prove that the free boundary is $C^1$.
\begin{proof}[Proof of Lemma \ref{C1 free}]
Applying Lemma~\ref{lemma h} to the functions $h_k$ introduced in \eqref{h_k} we obtain \eqref{tau_non_neg}.
As a consequence, for every $L > 0$ there exists $\widetilde{r} = \widetilde{r} (L) > 0$ such that
(recall that $K_u$ is defined in \eqref{fracture})
$$
\partial K_u \cap B^n_{\widetilde{r}} = \{ (x_1, \ldots, x_n) \in B^n_{\widetilde{r}} :  x_n = f_L (x_1, \ldots, x_{n-1}) \},
$$
for a suitable Lipschitz continuous function $f_L$ with Lipschitz constant $L$. 

Consider now a point $\hat x \in \partial K_u \cap B^n_{\widetilde{r}}$ and define the function $v_{\hat x}(x,y):=v(x-\hat x,y)$. Note that we can repeat the same argument (frequency formula and blow-up procedure) with $v_{\hat x}$ in place of $v$. Also, observe that since the function $\hat x \mapsto \Phi_{v_{\hat x}}(r)e^{Cr}$ is continuous for $r>0$ fixed,
the function $\hat x \mapsto \Phi_{v_{\hat x}}(0^+)$ is upper-semicontinuous (being the infimum over $r \in (0,\overline r_0)$ of continuous functions, cf. Proposition~\ref{9}). Hence, since $\Phi_{v_{\hat x}}(0^+) \in \{n+3\}\cup [n+4,\infty)$ (by Proposition~\ref{Phi}) and by assumption  $\Phi_{v_{0}}(0^+)=\Phi_{v}(0^+) = n+3$, we deduce 
that there exists $\hat r>0$ such that $\Phi_{v_{\hat x}}(0^+)=n+3$ for all $\hat x \in \partial K_u \cap B^n_{\hat{r}}$.

This implies that the previous argument can be repeated at every point in $\partial K_u \cap B^n_{\hat{r}}$,
and it follows that 
for any $L>0$ there exists  $\widetilde{r} (L) > 0$ such that $\partial K_u \cap B^n_{\widetilde{r} (L)}(\hat x)$
has Lipschitz constant $L$ for any point $\hat x \in \partial K_u \cap B^n_{\hat{r}}$.
Since $L > 0$ can be made arbitrarily small and the radius $\widetilde{r} (L) > 0$  is independent of $\hat x$, this implies that the free boundary is $C^1$ in a neighborhood of the origin.
\end{proof}

\begin{lemma}
Let the assumptions of Theorem~\ref{opt_reg_theorem} be satisfied, 
suppose that $(0, 0)$ is a free boundary point for $u$,
and that $u(x,0)\geq 0$ in $B_{\overline r_0}^n$.
Let $v$ be given by \eqref{v intro}, and let $\Phi_v$ be as in Proposition~\ref{9}.
Assume that $\Phi_v (0^+) = n+3$.
Then the free boundary is $C^{1, \alpha}$ near $(0, 0)$, for some $\alpha \in (0, 1)$.
\end{lemma}

\begin{proof}
We start by observing that the function $h (x, y) := \partial_{x_n} v (x, y)$ satisfies 
$$
 \partial_y h (x, 0) = 2 g'' (2 v (x,0)) h (x, 0) \qquad \text{ if } v (x, 0) > 0.
$$
Therefore,  by \cite{CS},
\begin{equation} \label{eq h}
( \Delta^{\frac{1}{2}} h (\cdot, 0) ) (x) = 2 g'' (2 v (x,0)) h (x, 0) + \partial_y \overline{h} (x, 0) - \partial_y h (x, 0) \qquad \text{ if } v (x, 0) > 0,
\end{equation}
where $\overline{h}$ is the harmonic extension of $h (\cdot, 0)$ to $\mathbb{R}^n \times (0, \infty)$.
Note that $\overline{h} - h$ is smooth near $\{ y = 0 \}$, 
since it is harmonic in $\mathbb{R}^n \times (0, A)$
with zero boundary condition on $\{ y = 0 \}$.
For every $0 < r \ll 1$, set 
$$
h_{r} (x) : = \frac{r}{d_r} h (rx, 0), \qquad x \in B^n_1, 
$$
where $d_r$ is given by \eqref{defin vr}.
From \eqref{eq h} it follows that, if $v (rx, 0) > 0$,
\begin{align*}
&( \Delta^{\frac{1}{2}} h_{r} (\cdot, 0)) (x) = \frac{r^2}{d_r} ( \Delta^{\frac{1}{2}} h (\cdot, 0) ) (r x) \\
&= \frac{r^2}{d_r} \left[ 2  g'' (2 v ( r x,0)) h (r x , 0) 
+ \partial_y \overline{h} (r x, 0) - \partial_y h (r x, 0) \right]  
= : F (x).
\end{align*}
Since $v$ and $h$ are bounded, we have
$$
| F | \leq C\frac{r^2}{d_r} (1+\| g'' \|_{L^{\infty}}),
$$
for some positive constant $C$.
Note also that, for $r$ sufficiently small, $h_r \geq 0$ in $B^n_1$ thanks to \eqref{tau_non_neg}.
Therefore, using the fact that $h_r = h_r^+$ in $B^n_1$ we obtain
$$
( \Delta^{\frac{1}{2}} h^+_{r} (\cdot, 0)) (x) \geq ( \Delta^{\frac{1}{2}} h_{r} (\cdot, 0)) (x) = F (x) \quad 
\text{ if } v (rx, 0) > 0.
$$
Moreover, by definition of $h_r$ we have $h_r (x, 0)= (\partial_{e_n} v_r ) (x, 0)$, 
where $v_r$ is defined in \eqref{defin vr}. Therefore, thanks to \eqref{convunif},
$$
h_r \to \partial_{x_n} v_{\infty} \quad \text{ uniformly in } B^n_{2/3},
$$
as $r\to0^+$.
Recalling \eqref{case a}, it follows that for $r$ small enough
$$
\sup_{B^n_{1/2}} h^+_r  = \sup_{B^n_{1/2}} h_r  \geq 1. 
$$
Let now $i \in \{ 1, \ldots, n-1\}$, $\tau_i := \frac{\mathbf e_n +  \mathbf e_i}{\sqrt{2}}$.
We can repeat the same argument used for $h$ for the function $h_{i} (x, y):= \partial_{\tau_i} v (x, y)$, 
obtaining that the function $x \mapsto h^+_{i, r} (x, 0) := (r/d_r) h^+_i (rx, 0)$ satisfies
$$
\begin{cases}
( \Delta^{\frac{1}{2}} h^+_{i, r} (\cdot, 0)) (x) \geq F_i (x) & \text{ for every } x \in B^n_1 \text{ with } v (rx, 0) > 0, \\
h^+_{i, r} (x, 0) =  0 & \text{ for every } x \in B^n_1 \text{ with } v (rx, 0) = 0,
\end{cases}
$$
with 
$$
| F_i | \leq  C\frac{r^2}{d_r} (1+\| g'' \|_{L^{\infty}}), \qquad 
\sup_{B^n_{1/2}} h^+_{i,r}  \geq 1. 
$$
Since $r^2/d_r \to 0$, for $r$ sufficiently small we can apply \cite[Theorem 1.6]{ROS}
to the nonnegative functions $h^+_{r} (\cdot, 0)$ and $h^+_{i,r} (\cdot, 0)$. 
We then obtain that the ratio $h^+_{i, r}/h^+_r$ is $C^{0, \alpha}$ in $B^n_{1/2}$, 
for some $\alpha \in (0, 1)$.
Since equalities $h^+_{i, r} = h_{i,r}$ and $h^+_{r} = h_{r}$ hold true in $B^n_{1/2}$, 
it follows that $h_i /h$ is of class $C^{0, \alpha}$ is a neighborhood of the origin. 
Let now $f$ be the function given by Lemma \ref{C1 free}.
Since 
$$
\frac{h_i}{h} = \frac{1}{\sqrt{2}} + \frac{1}{\sqrt{2}} \, \frac{ \partial_{x_i} v (x, 0)}{\partial_{x_n} v (x, 0)}, 
$$
and 
$$
\nabla f = -
\left( \frac{ \partial_{x_1} v }{\partial_{x_n} v } , \ldots, \frac{ \partial_{x_{n-1}} v}{\partial_{x_n} v} \right),  
$$
this implies that $f$ is $C^{1, \alpha}$ in a neighborhood of the origin.
\end{proof}
\end{section}

\begin{section}{Acknowledgments}

\noindent
L. Caffarelli is supported by NSF grant DMS-1160802.
\smallskip

\noindent
F. Cagnetti would like to thank Massimiliano Morini for bringing this problem to his attention. 
F. Cagnetti was partially supported by FCT through UT Austin$|$Portugal program, 
and by the EPSRC under the Grant EP/P007287/1 
``Symmetry of Minimisers in Calculus of Variations''.
F. Cagnetti acknowledges the hospitality of the Mathematics Department at UT Austin and of the FIM institute at ETH Z\"urich, where most of this work was done.

\smallskip

\noindent
A. Figalli is supported by the ERC Grant ``Regularity and Stability in Partial Differential Equations (RSPDE)''.

\smallskip

\noindent
The authors declare to have no conflict of interest.

\end{section}


\begin{thebibliography}{30}

\bibitem{A} M. Allen:
Separation of a lower dimensional free boundary in a two-phase problem. 
\textit{Math. Res. Lett.} \textbf{19} (2012), no. 5, 1055--1074. 

\bibitem{ALP} M. Allen, E. Lindgren, A. Petrosyan: 
The two-phase fractional obstacle problem. 
\textit{SIAM J. Math. Anal.} \textbf{47} (2015), no. 3, 1879--1905. 

\bibitem{AP} M. Allen, A. Petrosyan: 
A two-phase problem with a lower-dimensional free boundary. 
\textit{Interfaces Free Bound.} \textbf{14} (2012), no. 3, 307--342. 

\bibitem{ACFS} M. Artina, F. Cagnetti, M. Fornasier, F. Solombrino: 
Linearly Constrained Evolutions of Critical Points and an Application to Cohesive Fractures.
\textit{Math. Models Methods Appl. Sci.} \textbf{27} (2017), no. 2, 231--290. 

\bibitem{AC} I. Athanasopoulos, L. Caffarelli: 
Optimal regularity of lower dimensional obstacle problems. 
\textit{Zap. Nauchn. Sem. S.-Peterburg. Otdel. Mat. Inst. Steklov} (POMI) 310 (2004), 
\textit{Kraev. Zadachi Mat. Fiz. i Smezh. Vopr. Teor. Funkts} \textbf{35} [34], 49--66, 226; 
translation in \textit{J. Math. Sci. (N. Y.)} \textbf{132} (2006), no. 3, 274--284.

\bibitem{ACS} I. Athanasopoulos, L. Caffarelli, S. Salsa: 
The structure of the free boundary for lower dimensional obstacle problems. {\em Amer. J. Math.} {\bf 130} (2008), no. 2, 485--498. 


\bibitem{Bar62} G.I. Barenblatt: The mathematical theory of
equilibrium cracks in brittle fracture. 
\textit{Adv. Appl. Mech.} \textbf{7} (1962), 55--129.

\bibitem{CF} L.A. Caffarelli, A. Figalli: Regularity of solutions
to the parabolic fractional obstacle problem. 
\textit{J. Reine Angew. Math.} \textbf{680} (2013), 191--233.

\bibitem{CSS}
L.A. Caffarelli, S. Salsa, L. Silvestre:
Regularity estimates for the solution and the free boundary
of the obstacle problem for the fractional Laplacian.
{\em Invent. Math.} \textbf{171} (2008), 425--461.

\bibitem{CS}
L.A. Caffarelli, L. Silvestre:
An extension problem related to the fractional Laplacian.
{\em Comm. Partial Differential Equations} \textbf{32} (2007), no. 7-9, 1245--1260. 

\bibitem{C} F. Cagnetti: A vanishing viscosity approach to fracture growth
in a cohesive zone model with prescribed crack path.
\textit{Math. Models Methods Appl. Sci.} \textbf{18} (2008), no. 7, 1027--1071. 

\bibitem{CT} F. Cagnetti, R. Toader:
Quasistatic crack evolution for a cohesive zone model 
with different response to loading and unloading:
a Young measures approach. 
\textit{ESAIM Control Optim. Calc. Var.} \textbf{17} (2011), 1--27.

\bibitem{CLO} V. Crismale, G. Lazzaroni, G. Orlando:
Cohesive fracture with irreversibility: quasistatic evolution for a model subject to fatigue.
\textit{Math. Models Methods Appl. Sci.} \textbf{28} (2018), no. 7, 1371--1412.

\bibitem{DDMM} G. Dal Maso, A. DeSimone, M.G. Mora, M. Morini:
A vanishing viscosity approach to quasistatic evolution in plasticity with softening.
\textit{Arch. Ration. Mech. Anal.} \textbf{189} (2008), no. 3, 469--544.

\bibitem{DMFT} G. Dal Maso, G.A. Francfort, R. Toader:
Quasi-static evolution in brittle fracture: the case of bounded solutions.
\textit{Calculus of variations: topics from the mathematical heritage of E. De Giorgi},
245--266, Quad. Mat., 14, Dept. Math., Seconda Univ. Napoli, Caserta, 2004.

\bibitem{DMFT2} G. Dal Maso, G.A. Francfort, R. Toader:
Quasistatic crack growth in nonlinear elasticity.
\textit{Arch. Ration. Mech. Anal.} \textbf{176} (2005), no. 2, 165--225.

\bibitem{DT02} G. Dal Maso, R. Toader:
A model for the quasi-static growth of brittle fractures:
existence and approximation results. 
{\it Arch. Ration. Mech. Anal.\/} {\bf 162} (2002), 101-135.

\bibitem{DalToa02} G. Dal Maso, R. Toader:
A model for the quasi-static growth of brittle fractures
based on local minimization.
\textit{Math. Models Methods Appl. Sci.},
\textbf{12}/12 (2002), 1773--1799.
      
\bibitem{DMZ} G. Dal Maso, C. Zanini:
Quasi-static crack growth for a cohesive zone model 
with prescribed crack path.
\textit{Proc. Roy. Soc. Edinburgh Sect. A}, 
\textbf{137A} (2007), 253--279.

\bibitem{FKS}
E.B. Fabes, C.E. Kenig, R.P. Serapioni: 
The local regularity of solutions of degenerate elliptic equations. 
\textit{Commun. Partial Differ. Equations} \textbf{7}(1), (1982), 77--116. 

\bibitem{FrMa98} G.A. Francfort, J.-J. Marigo: 
Revisiting brittle fracture as an energy minimization problem.
{\it J. Mech. Phys. Solids\/} {\bf 46} (1998), 1319-1342.

\bibitem{Mielhand} A. Mielke: Evolution of rate-independent systems.
Handbook of differential equations, evolutionary equations, vol. 2,
C.M. Dafermos, E. Feireisl (eds.), 461--559 Elsevier, Amsterdam, 2005.

\bibitem{ROS} X. Ros-Oton, J. Serra:
Boundary regularity estimates for nonlocal elliptic equations in $C^1$ and $C^{1,\alpha}$ domains.
\textit{Ann. Mat. Pura Appl.} \textbf{196} (2017), 1637--1668.

\bibitem{Silv}
L. Silvestre: The regularity of the obstacle problem for a fractional power 
of the Laplace operator.
{\em Comm. Pure Appl. Math.} \textbf{60} (2007), no. 1, 67--112.


\end {thebibliography}
\end{document}